\documentclass[12pt]{amsart}
\usepackage{amsmath}
\usepackage{amssymb}
\usepackage{amsfonts}
\usepackage{amsthm}
\usepackage{mathrsfs}
\usepackage[all]{xy}
\usepackage{color}
\usepackage{multicol}
\usepackage{enumerate}
\usepackage{longtable}
\usepackage{array}

\usepackage[vcentermath,enableskew]{youngtab}
\usepackage{multirow}
\usepackage{tikz}
\usetikzlibrary{arrows}

\usepackage[colorlinks=true, pdfstartview=FitV, linkcolor=blue, citecolor=blue, urlcolor=blue]{hyperref}

\theoremstyle{plain}
\newtheorem*{theorem*}{Theorem}
\newtheorem{theorem}{Theorem}[section]
\newtheorem{lemma}[theorem]{Lemma}

\newtheorem{corollary}[theorem]{Corollary}
\newtheorem{proposition}[theorem]{Proposition}

\theoremstyle{definition}
\newtheorem{definition}[theorem]{Definition}
\newtheorem{remark}[theorem]{Remark}

\newtheorem{example}[theorem]{Example}

\newcommand{\C}{\mathbb{C}}

\newcommand{\mO}{\mathbb{O}}

\newcommand{\la}{\lambda}
\newcommand{\ttx}{\mathtt{x}}

\newcommand{\cB}{\mathcal{B}}
\newcommand{\cC}{\mathcal{C}}
\newcommand{\cF}{\mathcal{F}}
\newcommand{\cN}{\mathcal{N}}

\newcommand{\cP}{\mathcal{P}}
\newcommand{\cQ}{\mathcal{Q}}
\newcommand{\cS}{\mathcal{S}}
\newcommand{\cT}{\mathcal{T}}

\newcommand{\cZ}{\mathcal{Z}}

\DeclareMathOperator{\End}{End}

\DeclareMathOperator{\im}{im}

\DeclareMathOperator{\Type}{Type}
\DeclareMathOperator{\eType}{eType}

\DeclareMathOperator{\Sp}{Sp}

\DeclareMathOperator{\RS}{eRS}
\DeclareMathOperator{\Irr}{Irr}
\DeclareMathOperator{\Std}{Std}

\newcommand{\gl}{\mathfrak{gl}}

\newcommand{\mf}{\mathfrak}
\newcommand{\fN}{\mathfrak{N}}

\newcommand{\spp}{\mathfrak{sp}}

\newcommand{\y}{10}
\newcommand{\ya}{11}

\newcommand{\yd}{14}
\newcommand{\ye}{15}
\newcommand{\yf}{16}
\newcommand{\yg}{17}
\newcommand{\yh}{18}

\begin{document}
\title[Irreducible components of exotic Springer fibres II: RS algorithm]{Irreducible components of exotic Springer fibres II: Robinson-Schensted algorithms}

\author{Vinoth Nandakumar}
\address{V.~Nandakumar: School of Mathematics and Statistics, University of Sydney}
\email{vinoth.nandakumar@sydney.edu.au}

\author{Daniele Rosso}
\address{D.~Rosso: Department of Mathematics and Actuarial Science, Indiana University Northwest}
\email{drosso@iu.edu}

\author{Neil Saunders}
\address{N.~Saunders: Department of Mathematical Sciences, Old Royal Naval College,  University of Greenwich}
\email{n.saunders@greenwich.ac.uk}

\maketitle

\begin{abstract}
Kato's exotic nilpotent cone was introduced as a substitute for the ordinary nilpotent cone of type C with cleaner properties. The geometric Robinson-Schensted correspondence is obtained by parametrizing the irreducible components of the Steinberg variety (the conormal variety for the action of a semisimple group on two copies of its flag variety); in type A the bijection coincides with the classical Robinson-Schensted algorithm for the symmetric group. Here we give a combinatorial description of the bijection obtained by using the exotic nilpotent cone instead of ordinary type C nilpotent cone in the geometric Robinson-Schensted correspondence; we refer this as the "exotic Robinson-Schensted bijection". This is interesting from a combinatorial perspective, and not a naive extension of the type A Robinson-Schensted bijection. \end{abstract}
\tableofcontents

\section{Introduction}
The classical Robinson-Schensted correspondence is an algorithmic bijection 
$$
S_n \overset{\sim}\longleftrightarrow \bigsqcup_{\lambda \in \cP_n} \Std(\lambda) \times \Std(\lambda),
$$
where $S_n$ is the symmetric group of degree $n$, $\cP_n$ denotes the set of partitions of $n$ and $\Std(\lambda)$ denotes {\it standard Young tableaux} of shape $\lambda$. This bijection has many rich combinatorial features and many applications in representation theory: for instance, the resulting subsets $S_{\lambda}$ indexed by $\cP_n$ recover the two-sided cells as defined by Kazhdan and Lusztig in \cite[Section 5]{KL} and so index the families of unipotent characters of $GL_{n}(q)$ \cite[Section 18]{L4}. \vspace{5pt}

In \cite{SteinRSK}, Steinberg gives a geometric construction of the Robinson-Schensted correspondence using Springer theory. Using Spaltenstein's theorem, \cite{Spa}, stating that the irreducible components of Type A Springer fibres and in bijection with standard Young tableaux, Steinberg parametrises the irreducible components of the following variety $\cZ$ (the so-called Steinberg variety) in two different ways.
$$\cZ:=\{(x,V_{\bullet},U_{\bullet}) \, | \, x \in \cN, \, V_{\bullet},U_{\bullet} \in \cF_{x}\}$$
Above $\cN$ denotes the nilpotent cone of $\mf{gl}_{n}$ and $\cF_{x}$ denotes the Springer fibre above $x \in \cN$. By matching up these two different descriptions, in \cite{Ste} Steinberg showed that the resulting bijection coincides with the Robinson-Schensted correspondence. \vspace{5pt}

While the bijection in \cite{SteinRSK} is in fact defined for an arbitrary semisimple group, in types B, C, and D, the resulting algorithm is more complicated than the Robinson-Schensted bijections corresponding to the Weyl groups in those types; van Leeuwen, \cite{leeuwen}, describes it, building on earlier work of Spaltenstein (Section II.6 of \cite{irredspa}) describing irreducible components of Springer fibers in those types. In this paper, we examine the `exotic Robinson-Schensted bijection' - the analogous algorithm obtained using the geometry of Kato's exotic nilpotent cone as a substitute for the ordinary nilpotent cone of type $C$. The resulting combinatorial algorithm is more tractable, but again is different from the Robinson-Schensted algorithm arising from the Weyl group of type C. This builds on our previous paper, \cite{NRS16}, parametrizing irreducible components of exotic Springer fibers. Note that this is different to the exotic Robinson-Schensted correspondence constructed by Henderson and Trapa in \cite{HT}.
\vspace{5pt}

Let $\cN(\mf{gl}_{2n})$ be the nilpotent cone for $GL_{2n}$ and let $\cN(\cS)=\cN(\mf{gl}_{2n}) \cap \cS$ be the Hilbert nullcone of the $Sp_{2n}$ representation $\C^{2n} \oplus \cS$, where $\C^{2n}$ is the natural representation and $\cS$ is the $Sp_{2n}$-invariant complement of $\mf{sp}_{2n}$ in $\mf{gl}_{2n}$; Kato's exotic nilpotent cone for $Sp_{2n}$ is the variety $\mf{N}=\C^{2n} \times \cS$. In \cite{Kat}, Kato constructs an {\it exotic} Springer correspondence, and showed that the $Sp_{2n}$-orbits on $\mf{N}$ are in bijection with the bipartitions of $n$. In subsequent work, many other Springer theoretic results have been extended to the exotic setting - intersection cohomology of orbit closures, (see Achar and Henderson, \cite{AH}, and Shoji-Sorlin, \cite{ss}), theory of special pieces (see Achar-Henderson-Sommers, \cite{special}), and the Lusztig-Vogan bijection (see \cite{vn}). In many respects, the exotic nilpotent cone behaves more nicely than the ordinary nilpotent cone of type C, and the present paper is another illustration of this. \vspace{5pt}

Let $\pi: \widetilde{\mf{N}} \longrightarrow \mf{N}$ be the exotic Springer resolution as defined in \cite{Kato5}. In \cite{NRS16}, we showed that the irreducible components of the fibres of $\pi$ are in bijection with {\it standard Young bitableaux}, using Spaltenstein's techniques from \cite{Spa}. This allowed us to define an analogous {\it exotic Steinberg variety} whose irreducible components are parametrised in two separate ways: one way by elements of the Weyl group $W(C_n)=C_2 \wr S_n$, and the other by pairs of standard Young bitableaux. Hence, by matching up these two descriptions, this gives us a bijection: 
\begin{equation} \label{equation:mainbijection}
W(C_n) \overset{\sim}\longleftrightarrow \bigsqcup_{(\mu,\nu) \in \cQ_n} \cT(\mu,\nu) \times \cT(\mu,\nu),
\end{equation}
between $W(C_n)$ and pairs of standard Young bitableaux $ \cT(\mu,\nu)$ as we run through all bipartitions $\cQ_n$ of $n$. In this paper, we will give an explicit combinatorial description of this bijection.

The organisation of the paper is as follows: 
\begin{itemize}
\item In Section \ref{section:notation} we recall facts about the exotic nilpotent cone that we will need.
\item In Section \ref{section:algorithm}, which is mostly independent of the previous section, we define the exotic Robinson-Schensted bijection, a bijection between the Weyl group $W(C_n)$ and pairs of standard Young bitableaux. We provide the insertion and reverse bumping algorithms; these are interesting new algorithms and not a naive extension of the usual RS correspondence.
\item In Section \ref{section:generichyperplanes} we examine the exotic Springer fibres, understanding the restriction of exotic Jordan types at pairs of generic points of the fibre.  
\item In Section \ref{section:algorithmconstruction} we use the results from Section \ref{section:generichyperplanes} and \cite{NRS16} to construct the reverse bumping algorithm from the geometry of exotic Springer fibres and thus prove the main theorem. 


\end{itemize}

\subsection{Acknowledgements}
We are deeply grateful to Anthony Henderson for many useful discussions. We would also like to thank George Lusztig, Arun Ram, Nathan Williams and Igor Pak for useful conversations about the combinatorics of the Robinson-Schensted Correspondence. The third author is indebted to Donna Testerman and the \'Ecole Polytechinque F\'ed\'erale de Lausanne for their support where much of this work was carried out, as well as the Heilbronn Institute for Mathematical Research and City, University of London. 

\section{Background} \label{section:notation}
\subsection{Partitions} We recall some standard combinatorial definitions which we will need. We closely follow the notation of \cite[Section 2]{NRS16} in this exposition. 

\begin{definition}[Partitions and Bipartitions]
Let $n$ be a non-negative integer. A {\it partition} of $n$ is a sequence of positive integers $\lambda=(\lambda_1, \ldots ,\lambda_k)$ such that $\lambda_1 \geq \ldots \geq \lambda_k$ and $\sum_{i=1}^{k}\lambda_i=n$. We write $\lambda \vdash n$ or $|\lambda|=n$ to denote that $\lambda$ is a partition $n$ and write $\ell(\lambda)=k$ to say that $\lambda$ has $k$ parts, or length $k$. \vspace{5pt}

A {\it bipartition} of $n$ is a pair of partitions $(\mu,\nu)$ such that both $\mu$ and $\nu$ are partitions and that $|\mu|+|\nu|=n$. We let $\cQ_n$ denote the set of bipartitions of $n$. Given a bipartition $(\mu,\nu)$ of $n$, we let $\lambda:=\mu+ \nu = (\mu_1 + \nu_1, \mu_2+\nu_2, \ldots )$ denote the corresponding partition of $n$ whose $i$-th part is the sum of the $i$-th parts of $\mu$ and $\nu$ respectively. \vspace{5pt}

Associated to a partition $\lambda$, we have a Young diagram consisting of $\lambda_i$ boxes on row $i$. We say that the Young diagram has shape $\lambda$. Similarly, we have a pair of Young diagrams associated to each bipartition. 
\end{definition}

\begin{definition}\label{definition:gammadeltasets}

Fix a bipartition $(\mu,\nu) \in \cQ_n$ and let $\lambda=\mu+\nu$ be the corresponding partition of $n$. Fix a positive integer $m \leq \ell(\lambda)$. 
We define the following sets: 
\begin{eqnarray*}
\Lambda_{m}	&=& \{1 \leq i \leq \ell(\lambda) \, | \, \lambda_i=\lambda_m\}, \\
\Gamma_{m}	&=& 	\{ 1 \leq i \leq \ell(\lambda) \, | \,  \mu_i=\mu_m\}, \\
\Delta_{m}		&=& \{1 \leq i \leq \ell(\lambda)\, | \,  \nu_i=\nu_m\}.
\end{eqnarray*}
Moreover, define $$\Delta_{\leq m}=\{ i \in \Delta_m \, | \, i \leq m \, \} \quad \text{and} \quad \Delta_{< m}=\{ i \in \Delta_m \, | \, i < m \, \} ,$$ with similar definitions for $\Gamma_m$ and $\Lambda_m$.  
\end{definition}

\begin{definition}[Standard Young Bitableaux]
Given a Young diagram of shape $\lambda \in \cP_n$, we obtain a standard Young tableau by filling in the boxes with the integers $1$ up to $n$ in such a way that the numbers are increasing along rows and down columns. Similarly a bitableau of shape $(\mu, \nu) \in \cQ_n$ is standard if every integer between $1$ to $n$ occurs exactly once, and the increasing condition is satisfied in each of the two tableaux. To match the conventions of \cite{NRS16} and \cite{AH}, we reverse the direction of the rows of the first tableau, so numbers are increasing along rows from right-to-left. We let $\cT(\mu,\nu)$ denote the set of standard Young bitableaux of shape $(\mu,\nu)$.
\end{definition}

\begin{example} \label{example:bitableau1}
An example of a standard Young bitableau of shape $((3,1), (2,2,1))$ is the following: 
$$T= \left( \young(631,::2),\young(47,58,9)\right).$$ However, for notational, convenience especially for when describing the algorithm, we will often draw this diagram as follows: 
\begin{center}
\begin{tikzpicture}[scale=0.5]
\draw[-,line width=2pt] (0,3.5) to (0,0.5);

\node[left] at (0,3) {$1$};
\node[left] at (-2,3) {$6$};
\node[left] at (-1,3) {$3$};
\node[left] at (0,2) {$2$};

\node[right] at (0,3) {$4$};
\node[right] at (1,3) {$7$};
\node[right] at (0,2) {$5$};
\node[right] at (1,2) {$8$};
\node[right] at (0,1) {$9$};
\node[left] at (-3,2) {$T=$};

\end{tikzpicture}
\end{center}
\end{example}
The positions of the numbers in the bitableau $T$ are described via the partition $\lambda$ of $n$; that is, a number $s$ is described as being in the $(i,j)$-th position of $T$ where $1 \leq i \leq \ell(\lambda)$ and $1 \leq j \leq \lambda_i$. Thus, for $T$ as above, $3$ is in the $(1,2)$-th position, $5$ is in the $(2,2)$-th position, and $9$ is in the $(3,1)$-th position.
\begin{definition}\label{def:truncT} For $T$ a standard Young bitableau and $1 \leq s \leq n$, define $T_s$ to be the truncated bitableau consisting of just the numbers $1$ up to $s$. By definition this remains a standard Young bitableau. If $T$ originally had shape $(\mu,\nu) \in \cQ_n$ then  $T_s$ has shape $(\mu^{(s)},\nu^{(s)}) \in \cQ_s$.  For a truncated tableau $T_s$ of shape $(\mu^{(s)},\nu^{(s)})$, define the sets $ \Lambda_{m}^{(s)}, \Gamma_{m}^{(s)}$ and $\Delta_{m}^{(s)}$ in a completely analogous way to Definition \ref{definition:gammadeltasets}.
\end{definition} 

\begin{example} For $T$ as in Example \ref{example:bitableau1}, we have $$T_5=  \left( \young(31,:2),\young(4,5)\right),$$ which has shape $((2,1),(1,1)) \in \cQ_5$.  \end{example}

\subsection{The exotic nilpotent cone}
Let $V \cong \mathbb{C}^{2n}$ be a vector space endowed with a symplectic form $\langle \cdot , \cdot \rangle$. Denote by $\Sp_{2n}(\mathbb{C})$ the corresponding symplectic group and $\mathfrak{sp}_{2n}(\mathbb{C})$ its Lie algebra. 

\begin{definition}We define $W(C_n)$ to be the Weyl group of Type $C_n$, which is isomorphic to the group of signed permutations. For $w \in W(C_n)$, we write $w=w_1 \ldots w_n$ where $w_j=a_j$ or $w_j=\bar{a}_j$, with $a_j\in\{1,\ldots,n\}$, for $j=1,\ldots,n$.
\end{definition}


\begin{definition}\label{def:embedding}
The natural embedding of $\Sp_{2n}$ in $\operatorname{GL}_{2n}=\operatorname{Aut}(V)$ induces an embedding of Weyl groups $ \iota:W(C_n) \hookrightarrow S_{2n}$ as follows: 
$$ \iota: w_1w_2\cdots w_n\mapsto \sigma_1\sigma_2\cdots\sigma_n\sigma_{n+1}\cdots \sigma_{2n}$$
where for $i=1,\ldots, n$,
$$ 
\sigma_i=\begin{cases}
n+1- a_{n+1-i} & \text{ if }w_i=a_i \\ 
n+a_i & \text{ if }w_i=\bar{a}_i 
\end{cases}\quad \text{and} \quad \sigma_{2n+1-i}=2n+1-\sigma_{i}.$$
The image of $W(C_n)$ under this embedding is $$\{ \sigma \in S_{2n} \, | \, \sigma(i)+\sigma(2n-i+1)=2n+1, \, \text{for all} \, i \}.$$ 
\end{definition}

\begin{definition} Define $\cS$ as follows, noting that 
$\gl_{2n}=\spp_{2n}\oplus \cS$ as $\Sp_{2n}(\mathbb{C})$-modules. $$ \cS=\{ x\in \End(V)~ | ~\langle xv,w\rangle -\langle v,xw\rangle =0, ~\forall v,w\in V \} $$

The \emph{exotic nilpotent cone} is the following singular variety $\mathfrak{N}= \C^{2n} \times \cN(\cS)$. It carries a natural $\Sp_{2n}(\mathbb{C})$-action: $$g\cdot (v, x)= (gv,gxg^{-1}),$$ for $g \in Sp_{2n}(\C)$ and $(v,x) \in \mathfrak{N}$.
\end{definition}

\begin{theorem}[{\cite[Thm 6.1]{AH}}] The orbits of $\Sp_{2n}(\mathbb{C})$ on $\fN$ are in bijection with $\cQ_n$. 
More precisely, given a bipartition $(\mu,\nu)\in\cQ_n$, the corresponding orbit $\mO_{(\mu,\nu)}$ contains the point $(v,x)$ if and only if there is a `normal' basis of $V$ given by $$\{v_{ij}, v_{ij}^{*} \, | \, 1 \leq i \leq \ell(\mu+\nu), 1 \leq j \leq \mu_i + \nu_i \},$$ with $\langle v_{ij},v^*_{i'j'}\rangle=\delta_{i,i'}\delta_{j,j'}$, $v=\sum_{i=1}^{\ell(\mu)}v_{i,\mu_i}$
and such that the action of $x$ on this basis is as follows: 
$$xv_{ij}=\begin{cases} v_{i,j-1}  & \mbox{if} \quad j \geq 2 \\ 0  &\mbox{if} \quad j=1 \end{cases} \quad \quad \quad xv_{ij}^{*}=\begin{cases} v_{i,j+1}^{*}  & \mbox{if} \quad j \leq \mu_i+\nu_i-1 \\ 0  &\mbox{if} \quad  j=\mu_i+\nu_i \end{cases}$$ in particular the Jordan type of $x$ is $(\mu+\nu)\cup(\mu+\nu)$. \end{theorem}

\begin{definition} If $(v,x)\in\mO_{(\mu,\nu)}\subset\fN$, we say that $(\mu,\nu)$ is the exotic Jordan type of $(v,x)$ and we denote that by $\eType(v,x)=(\mu,\nu)$. \end{definition}

\begin{definition}The flag variety for $Sp_{2n}(\mathbb{C})$, which we denote by $\mathcal{F}(V)$, is the variety consisting of all symplectic flags, that is sequences of subspaces $$ F_\bullet=(0=F_0\subseteq F_1\subseteq \cdots \subseteq F_n\subseteq \cdots \subseteq F_{2n-1}\subseteq F_{2n}=V)$$ 
where $\dim(F_i)=i$ and $F_i^\perp=F_{2n-i}$. 

Define the exotic Springer resolution $\pi: \widetilde{\fN} \twoheadrightarrow \fN$ as follows:  $$\widetilde{\mathfrak{N}} = \{(F_\bullet,(v,x))\in\cF(V)\times\fN~|~v\in F_n,~x(F_i)\subseteq F_{i-1}~\forall i=1,\ldots,2n\} $$
$$ \pi(F_\bullet,(v,x))=(v,x).$$
 \end{definition}

\vspace{5pt}


\begin{definition} \label{relative} Given two flags $F_{\bullet}, G_{\bullet} \in \cF(V)$, we say that $w(F_{\bullet}, G_{\bullet}) = w$ (i.e. the two flags are in relative position $w \in W(C_n)$) if there is a basis $\{v_{1}, \ldots, v_{n}, {v}_{\bar{n}}, \ldots, v_{\bar{1}}\}$ such that $\langle v_{i},v_{j} \rangle= \langle v_{\bar{i}},v_{\bar{j}} \rangle=0$ and $\langle v_{i},v_{\bar{j}} \rangle=\delta_{ij}$ such that for $1 \leq i,j,\leq  n$ we have 
\begin{eqnarray*}
F_{i} &=& \C \{ v_{n}, \ldots, v_{n-i+1}\} \quad \text{and} \quad F_{2n-i}=F_{i}^{\perp}, \quad \text{and}  \\ 
G_{j} &=& \C\{v_{w(n)}, \ldots, v_{w(n-j+1)}\} \quad \text{and} \quad G_{2n-j}=G_{j}^{\perp}
\end{eqnarray*}
\end{definition}
\subsection{Components of exotic Springer fibers}

Here we recall the results from \cite{NRS16} describing the irreducible components of exotic Springer fibers. 

\begin{definition} Given $(v,x) \in \mathbb{O}_{(\mu, \nu)}$, define the exotic Springer fibre $\cC_{(v,x)} = \pi^{-1}(v,x)$. Explicitly:
$$ \cC_{(v,x)} = \{(0 \subset F_1 \subset \ldots \subset F_{2n-1} \subset \C^{2n}) \, | \, \dim F_i =i, F_{i}^{\perp}=F_{2n-i}, v \in F_n, x(F_{i}) \subset F_{i-1} \} $$ \end{definition}

The main result of \cite{NRS16} was the following: 

\begin{theorem}\cite[Theorem 2.12]{NRS16} \label{theorem:irreduciblecomponents}
Let $(v,x) \in \mathbb{O}_{(\mu, \nu)}$, then there is an open dense subset $\cC_{(v,x)}^{\circ}\subset\cC_{(v,x)}$, and a surjective map $\Phi: \cC_{(v,x)}^{\circ} \longrightarrow \cT(\mu,\nu)$ which induces a bijection between irreducible components of $\cC_{(v;x)}$ and standard Young bitableaux of shape $(\mu,\nu)$:
\begin{eqnarray*}
\Irr \cC_{(v,x)} 			&\overset{\sim}\longleftrightarrow& \cT(\mu,\nu); \\ 
\overline{\Phi^{-1}(T)} 	& \longleftrightarrow&	T,
\end{eqnarray*}
These irreducible components all have the same dimension: $b(\mu,\nu)= |\nu| + \sum _{i \geq 1}(i-1)(\mu_i+\nu_i)$.  \end{theorem}
\begin{remark}
A standard bitableau of shape $(\mu,\nu)$ is the same thing as a \emph{nested} sequence of bipartitions ending at $(\mu,\nu)$ i.e. a sequence of bipartitions
$$(\varnothing, \varnothing), (\mu^{(1)}, \nu^{(1)}), \ldots, (\mu^{(n)},\nu^{(n)})=(\mu,\nu)$$
such that $(\mu^{(i+1)},\nu^{(i+1)})$ is obtained from $(\mu^{(i)},\nu^{(i)})$ by adding one box. The identification is given by tracing the order in which the boxes are added according to the increasing sequence of numbers $1,2,\ldots,|\mu|+|\nu|$.
\end{remark}
\begin{example}
The standard bitableau $\left(\young(32,:5),\young(1,4)\right)$ corresponds to the nested sequence
$$ \left(\varnothing,\varnothing\right),\left(\varnothing,\yng(1)\right),\left(\yng(1),\yng(1)\right),\left(\yng(2),\yng(1)\right),\left(\yng(2),\yng(1,1)\right),\left(\young(\hfil\hfil,:\hfil),\yng(1,1)\right).$$
\end{example}
\begin{remark}
The map $\Phi$ of Theorem \ref{theorem:irreduciblecomponents} is defined as follows, if $F_\bullet\in \cC_{(v,x)}$, then
$$\Phi(F_\bullet)=(\eType(v+F_i,x|_{F_i^{\perp}/F_i}))_{i=0}^n.$$
For $F_\bullet\in\cC_{(v,x)}^{\circ}$ we have then that $\Phi(F_\bullet)$ is a nested sequence of bipartitions, i.e. a standard bitableau.
\end{remark}
\begin{remark} From Travkin \cite[Thm 1 and Cor 1]{T}, or Achar-Henderson \cite[Thm 6.1]{AH}, we know that $\eType(v,x)=(\mu,\nu)$ if and only if the following holds (here $\Type$ denotes the partition corresponding to the Jordan type of the corresponding nilpotent endomorphism):
\begin{eqnarray*}
\Type(x,V) &=&(\mu_1+\nu_1, \mu_1+\nu_1, \mu_2+\nu_2, \mu_2+\nu_2, \ldots ), \quad \text{and} \\
\Type(x,V/\C[x]v) &=& (\mu_1+\nu_1, \mu_2+\nu_1, \mu_2+\nu_2, \mu_3+\nu_2, \ldots ). 
\end{eqnarray*} \end{remark}

A key idea in proving Theorem \ref{theorem:irreduciblecomponents} was the following construction. 

\begin{proposition}\label{proposition:genericw} Let $(v,x)\in\fN$, with $\eType(v,x)=(\mu,\nu)$ and let $(\mu',\nu')$ be a bipartition of $n-1$ obtained from $(\mu,\nu)$ by decreasing either $\mu_m$ or $\nu_m$ by $1$. We define the set $$ \cB_{(\mu',\nu')}^{(\mu,\nu)}:=\{W \subset \ker(x) \cap (\mathbb{C}[x]v)^{\perp}\, | \, \eType(v+W,x_{|W^{\perp}/W})=(\mu',\nu')\, \} .$$ 
For a generic point $W\in\cB_{(\mu',\nu')}^{(\mu,\nu)}$, we have $W=\mathbb{C}w$ where
\begin{equation} \label{eqn:genericw}
w=\sum_{i=1}^{m} \alpha_{i}v_{i,1}+\beta_{i}v_{i,\lambda_i}^{*},
\end{equation} 
satisfies the following:
\begin{itemize}
\item  if $\mu'_m=\mu_m-1$, then $\sum_{i \in\Delta_{\leq m}}\beta_i=0$;

\item if $\nu'_m=\nu_m-1$, then $\sum_{i \in \Delta_{\leq m}}\beta_i \neq 0$.
\end{itemize}
\end{proposition}
\begin{proof} Suppose first that  that $(\mu', \nu')$ is obtained from $(\mu,\nu)$ by decreasing $\mu_m$ by $1$. By \cite[Proposition 4.18]{NRS16}, this implies that  $W^{\perp} \supseteq (x^{\lambda_m-1})^{-1}(\C[x]v+W)$ but $W^{\perp} \not\supseteq (x^{\lambda_m})^{-1}(\C[x]v+W)$, which requires  $\sum_{i \in\Delta_{\leq m}}\beta_i=0$. \vspace{5pt}

Now suppose that $(\mu', \nu')$ is obtained from $(\mu,\nu)$ by decreasing $\nu_m$ by $1$. Then again \cite[Proposition 4.18]{NRS16}, this implies that 
$$W^{\perp} \supseteq (x^{\mu_{m''}+\nu_m-1})^{-1}(\C[x]v+W) \quad \text{but} \quad W^{\perp} \not\supseteq (x^{\mu_{m+1}+\nu_m})^{-1}(\C[x]v+W),$$ which requires that $\sum_{i \in \Delta_{\leq m}}\beta_i \neq 0$ if $\max \Gamma_m=\max \Delta_m=m$ (for otherwise, we would be decreasing $\mu_m$ by $1$). There are no conditions on the $\beta_i$ if $\max \Gamma_m > \max \Delta_m=m$ or if $\mu_m=0$, but even in that case, the equation gives an open condition which is true generically. 
\end{proof}
\begin{remark} \label{corollary:genericw}
In the case where $\mu'_m=\mu_m-1$, and $\nu_{m-1} > \nu_m$, this implies that $\beta_m=0$ in the expression \eqref{eqn:genericw} for the spanning vector $w$ of $W$. \end{remark}

We record the below Proposition for future use in Section \ref{section:generichyperplanes}. 

\begin{proposition} \label{proposition:perpendicular1}
Let $(\mu',\nu')$ obtained from $(\mu,\nu)$ by decreasing $\mu_m$ or $\nu_m$ by $1$ and let $W, X\in\cB_{(\mu',\nu')}^{(\mu,\nu)}$ be two generic points. Then $X \subset W^{\perp}$ if and only if  
\begin{enumerate}[(a)]
\item $\lambda_m \geq 2$; or 
\item $\lambda_m=\mu_m=1$, $\nu_{m-1} > \nu_m=0$ and $\mu'_m=\mu_m-1$.
\end{enumerate}
\end{proposition}

\begin{proof}
Let $W=\mathbb{C}w$ and $X=\mathbb{C}\mathtt{x}$ with 
\begin{equation} \label{eqn:genericwx}
w=\sum_{i=1}^{m} \alpha_{i}v_{i,1}+\beta_{i}v_{i,\lambda_i}^{*} \quad \text{and} \quad \mathtt{x}=\sum_{i=1}^{m} \gamma_{i}v_{i,1}+\delta_{i}v_{i,\lambda_i}^{*}.
\end{equation}

Then 
\begin{equation}\label{eqn:xperpw} \langle w,\mathtt{x} \rangle =\sum_{i\leq m ~|~\la_i=1}(\alpha_i\delta_i-\beta_i\gamma_i).\end{equation}
If $\la_m\geq 2$, then the above sum is an empty sum, so $\langle w,x \rangle =0$ always.
%
In the case where $\lambda_m=\mu_m=1$ and $\nu_{m-1} > \nu_m=0$ with $\mu'_m=\mu_m-1$, by Corollary \ref{corollary:genericw} we have $\beta_m=\delta_m=0$ for generic points, so 
$$\langle w, \mathtt{x} \rangle=\alpha_m\cdot 0 - \gamma_m \cdot 0=0.$$ 

In all other cases the equation \eqref{eqn:xperpw} will give generically a nonzero result.
\end{proof}
\begin{corollary} \label{corollary:notperp}
With $W$ and $X$ as in Proposition \ref{proposition:perpendicular1}, we have that $X \not\subseteq W^{\perp}$ if and only if 
\begin{enumerate}[(a)]
\item $\lambda_m=\mu_m=1$ and $\nu_{m-1}=0$;
\item $\lambda_m=\nu_m=1$ (which implies $\mu_m=0$). 
\end{enumerate}
\end{corollary}

\begin{corollary} \label{corollary:uniquegeneric}
Let $X$ and $W$ as above, and suppose $(\mu',\nu')$ is obtained from $(\mu,\nu)$ by decreasing $\mu_1$ by $1$. Then $X=W$. 
\end{corollary}
\begin{proof}
In this case, since $m=1$, by Proposition \ref{proposition:genericw} we have that $\beta_1=0$ in \eqref{eqn:genericw}, hence the variety  $\cB_{(\mu',\nu')}^{(\mu,\nu)}$ consists of the single point $\{\mathbb{C}v_{1,1}\}=W=X$.
\end{proof}
\section{The Algorithm} \label{section:algorithm}
In this section we present the Exotic Robinson-Schensted algorithm in both directions, a `reverse bumping' direction from bitableaux to Weyl group elements and then an `insertion' algorithm in the other direction. 

\subsection{Reverse bumping Algorithm} \label{section:bumpingalgorithm}
The algorithm that we present now takes as an input a pair of standard Young bitableaux and produces an element in the Weyl group $W(C_n)$. Let $T$ and $R$ be two standard Young bitableaux. We will think of $T$ as the ``insertion'' bitableau and $R$ as the ``recording'' bitableau and produce a word $\RS(T,R) \in W(C_n)$ as follows.

\begin{definition}Let $1\leq s\leq n$ and let $T$ be a bitableau that is increasing going away from the center and going down (standard condition) and does not contain the number $s$. An \emph{available position} for $s$ in $T$ is a position $(i,j)$ in $T$ such that the number in that position is smaller than $s$ and if you replace that number with $s$, the increasing standard condition is still satisfied. This is equivalent to saying that $(i,j)$ is a corner of the tableau $T_s$ from Definition \ref{def:truncT}.
\end{definition}
\begin{example}
Let $s=13$, $T=\left(\young(\y31,\yd65,\ye\ya9);\young(4\yf,8\yg,\yh)\right)$, then the available positions for $13$ are the ones containing the numbers $10$, $11$ and $8$.
\end{example}

\begin{definition}
Given a bitableau $T$, of shape $(\mu,\nu)$, let $\mathscr{R}$ be the set of rows of the two tableaux that comprise $T$, corresponding to the parts of the partitions $\mu$ and $\nu$. We number the rows as follow,
$$ \mu_1=1,\quad \nu_1=2,\quad \mu_2=3,\quad \nu_2=4,\quad \mu_3=5,\ldots $$
So the first row of the left tableau is first, followed by the first row of the right tableau, followed by the second row of the left tableau and so on.
\end{definition}

\subsubsection*{The Algorithm}
Let $T,R$ be two standard bitableaux with $n$ boxes. 
\begin{enumerate}
\item Start with $k=n$.
\item Find the position in $R$ containing $k$. Let $s$ be the number in the same position in the bitableau $T$ and $m\in\mathscr{R}$ be the row in which $s$ appears. Let $R'$ be the bitableau obtained from $R$ by removing $k$ and $T'$ the bitableau obtained from $T$ by removing $s$.
\item If $m=1$, set $w(k)=s$. If $k=1$ stop: the algorithm has ended, otherwise, return to (2) with $k$ replaced by $k-1$.
\item If $m\neq 1$, consider all the available positions for $s$ in $T'$ that are in rows $\{r\in\mathscr{R}~|~r\geq m-1\}$.
\item If there are no available positions in those rows, let $w(k)=\bar{s}$ and return to (2) with $k$ replaced by $k-1$.
\item Otherwise, if there are available positions, let $(i,j)$ be the one in the smallest numbered row (notice that there can be at most one available position per row). Let $T'$ be the bitableau obtained by replacing the number in position $(i,j)$ with $s$, and let $s'$ be the number that has been displaced. Let $m\in\mathscr{R}$ be the row where $s'$ was. Return to (3) with $s$ replaced by $s'$.
\end{enumerate}

\begin{example}
Let $T=\left(\young(1,3,7); \young(24,56) \right)$, $R=\left(\young(2,5,6); \young(13,47) \right)$. We compute what element in $W(C_7)$ corresponds to this pair of bitableaux. For ease of notation we simply write the contents of the tableaux without the boxes in an array divided by the wall. 

$$
\begin{array}{cccccccc}

\begin{tikzpicture}[scale=0.5]
\draw[-,line width=2pt] (0,3.5) to (0,0.5);
\node[left] at (0,1) {$7$};
\node[left] at (0,2) {$3$};
\node[left] at (0,3) {$1$};
\node[right] at (0,3) {$2$};
\node[right] at (1,3) {$4$};
\node[right] at (0,2) {$5$};
\node[right] at (1,2) {$6$};

\node[left] at (-1,2) {$T=$};

\draw[-,line width=2pt] (0,-3.5) to (0,-0.5);
\node[left] at (0,-1) {$2$};
\node[left] at (0,-2) {$5$};
\node[left] at (0,-3) {$6$};
\node[right] at (0,-2) {$4$};
\node[right] at (1,-2) {$7$};
\node[right] at (0,-1) {$1$};
\node[right] at (1,-1) {$3$};
\node[left] at (-1,-2) {$R=$};

\draw[->] (1,0) -- (2,0);
\node[above] at (1.5,0) {\tiny{$w(7)=1$}};
\end{tikzpicture}
&

\begin{tikzpicture}[scale=0.5]
\draw[-,line width=2pt] (0,3.5) to (0,0.5);
\node[left] at (0,1) {$7$};
\node[left] at (0,2) {$6$};
\node[left] at (0,3) {$2$};
\node[right] at (0,3) {$3$};
\node[right] at (1,3) {$4$};
\node[right] at (0,2) {$5$};

\draw[-,line width=2pt] (0,-3.5) to (0,-0.5);
\node[left] at (0,-1) {$2$};
\node[left] at (0,-2) {$5$};
\node[left] at (0,-3) {$6$};
\node[right] at (0,-2) {$4$};

\node[right] at (0,-1) {$1$};
\node[right] at (1,-1) {$3$};

\draw[->] (1,0) -- (2,0);
\node[above] at (1.5,0) {\tiny{$w(6)=\bar{5}$}};
\end{tikzpicture}

&
\begin{tikzpicture}[scale=0.5]
\draw[-,line width=2pt] (0,3.5) to (0,0.5);

\node[left] at (0,2) {$6$};
\node[left] at (0,3) {$2$};
\node[right] at (0,3) {$3$};
\node[right] at (1,3) {$4$};
\node[right] at (0,2) {$7$};

\draw[-,line width=2pt] (0,-3.5) to (0,-0.5);
\node[left] at (0,-1) {$2$};
\node[left] at (0,-2) {$5$};

\node[right] at (0,-2) {$4$};

\node[right] at (0,-1) {$1$};
\node[right] at (1,-1) {$3$};

\draw[->] (1,0) -- (2,0);
\node[above] at (1.5,0) {\tiny{$w(5)=2$}};
\end{tikzpicture}

&
\begin{tikzpicture}[scale=0.5]
\draw[-,line width=2pt] (0,3.5) to (0,0.5);

\node[left] at (0,3) {$4$};
\node[right] at (0,3) {$3$};
\node[right] at (1,3) {$6$};
\node[right] at (0,2) {$7$};

\draw[-,line width=2pt] (0,-3.5) to (0,-0.5);
\node[left] at (0,-1) {$2$};

\node[right] at (0,-2) {$4$};

\node[right] at (0,-1) {$1$};
\node[right] at (1,-1) {$3$};
\draw[->] (1,0) -- (2,0);
\node[above] at (1.5,0) {\tiny{$w(4)=\bar{7}$}};
\end{tikzpicture}

&
\begin{tikzpicture}[scale=0.5]
\draw[-,line width=2pt] (0,3.5) to (0,0.5);

\node[left] at (0,3) {$4$};
\node[right] at (0,3) {$3$};
\node[right] at (1,3) {$6$};

\draw[-,line width=2pt] (0,-3.5) to (0,-0.5);
\node[left] at (0,-1) {$2$};

\node[right] at (0,-1) {$1$};
\node[right] at (1,-1) {$3$};
\draw[->] (1,0) -- (2,0);
\node[above] at (1.5,0) {\tiny{$w(3)=4$}};
\end{tikzpicture}

&

\begin{tikzpicture}[scale=0.5]
\draw[-,line width=2pt] (0,3.5) to (0,0.5);
\node[left] at (0,3) {$6$};
\node[right] at (0,3) {$3$};
\draw[-,line width=2pt] (0,-3.5) to (0,-0.5);
\node[left] at (0,-1) {$2$};
\node[right] at (0,-1) {$1$};
\draw[->] (1,0) -- (2,0);
\node[above] at (1.5,0) {\tiny{$w(2)=6$}};
\end{tikzpicture}

&
\begin{tikzpicture}[scale=0.5]
\draw[-,line width=2pt] (0,3.5) to (0,0.5);
\node[right] at (0,3) {$3$};
\draw[-,line width=2pt] (0,-3.5) to (0,-0.5);
\node[right] at (0,-1) {$1$};
\draw[->] (1,0) -- (2,0);
\node[above] at (1.5,0) {\tiny{$w(1)=\bar{3}$}};

\end{tikzpicture}
\end{array} 
$$

So in this case we have $\RS(T,R) = w=\bar{3}64\bar{7}2\bar{5}1$. \vspace{5pt}

We explain in more detail the first step of this algorithm which yields $w(7)=1$. At the first stage, $7$ is in the $(2,3)$ position of $R$ and $6$ is in the corresponding position of $T$. The following table shows where $6$ and then subsequent numbers move to according to the rules; at each stage, we show the restricted tableau defined by the number that is being moved:
\vspace{10pt}

\begin{center}
\begin{tabular}{ m{2cm}  m{3cm}  m{5cm}  b{5cm} }
$(s,m)$ & Tableau & Rule and Action \\
\hline
$(6,2)$
&
\begin{tikzpicture}[scale=0.5]
\draw[-,line width=2pt] (0,3.5) to (0,1.5);
\node[left] at (0,2) (3){$3$};
\node[left] at (0,3) {$1$};
\node[right] at (0,3) {$2$};
\node[right] at (1,3) {$4$};
\node[right] at (0,2) {$5$};
\node[right] at (1,2) (6){$6$};
\draw (6) edge[thick, color=blue, ->, out=270, in=270] (3);
\node[left] at (-1,2) {$T_6=$};
\end{tikzpicture} 

& The smallest available position is in the box occupied by $3$ \\
\hline
$(3,2)$ 
&
\begin{tikzpicture}[scale=0.5]
\draw[-,line width=2pt] (0,3.5) to (0,1.5);
\node[left] at (0,2) (3){$3$};
\node[left] at (0,3) {$1$};
\node[right] at (0,3)(2) {$2$};
\draw (3) edge[thick, color=blue, ->, out=0, in=270] (2);
\node[left] at (-1,2) {$T_3=$};
\end{tikzpicture}
& 
The smallest available position in the box occupied by $2$\\
\hline
$(2,1)$ 
&
\begin{tikzpicture}[scale=0.5]
\draw[-,line width=2pt] (0,3.5) to (0,2.5);
\node[left] at (0,3) (1){$1$};
\node[right] at (0,3)(2) {$2$};
\draw (2) edge[thick, color=blue, ->, out=270, in=300] (1);
\node[left] at (-1,3) {$T_2=$};
\end{tikzpicture}
& 
The smallest available position is in the box occupied by $1$ \\
\hline
$(1,1)$ 
&
\begin{tikzpicture}[scale=0.5]
\draw[-,line width=2pt] (0,3.5) to (0,2.5);
\node[left] at (0,3) (1){$1$};
\node[right] at (-2,4) (2) {};
\draw (1) edge[thick, color=blue, ->, out=145, in=300] (2);
\node[left] at (-1,3) {$T_1=$};
\end{tikzpicture}
& Algorithm stops here, $w(7) = 1$
\end{tabular}
\end{center}

\vspace{5pt}

\begin{remark} A similar calculation shows that $\RS(R,T) =w^{-1}=75\bar{1}3\bar{6}2\bar{4}$. The fact that exchanging the tableaux gives the inverse element of the Weyl group is a feature of the ordinary Robinson-Schensted Correspondence in Type A and is true also in our setting. One can deduce this from Theorem \ref{theorem:main}. \end{remark}
\end{example}

\subsection{ Insertion Algorithm}
We present the algorithm that takes a signed permutation word in $W(C_n)$ and produces a pair of standard bitableaux. Let $w=w_1\ldots w_n$ where $w_i \in \{1, \ldots, n \} \cup \{\bar{1}, \ldots, \bar{n}\}$. 

\vspace{5pt}

\begin{definition} Let $1\leq s\leq n$ and let $T$ be a bitableau that is increasing going away from the center and going down (standard condition) and does not contain the number $s$. An \emph{insertable position} for $s$ is a position $(i,j)$, which is either outside of $T$ but adjacent to a box of $T$, or in $T$ such that the number in that position is bigger than $s$, and such that if you insert $s$ there (possibly replacing a number), the resulting shape is still a bipartition and the increasing standard condition is still satisfied. 
\end{definition}
\begin{example}
Let $s=13$, $T=\left(\young(\y31,\yd65,\ye\ya9);\young(4\yf,8\yg,\yh)\right)$, then the insertable positions for $13$ are to the left of the number $10$, under the number $9$, in addition to the ones containing the numbers $14$, $16$ and $18$.
\end{example}

Remember that the rows of a bitableau are ordered as in Definition 3.3.
\subsubsection*{The Algorithm}
Let $w=w_1w_2\ldots w_n$, with $w_i\in \{1,2,\ldots,n\}\cup\{\bar{1},\bar{2},\ldots,\bar{n}\}$ be a signed permutation with $n$ letters.  
\begin{enumerate}
\item We set $T,R$ to be two empty standard bitableaux, and we start with $k=1$.
\item If $w_k=s\in\{1,2\ldots,n\}$, add a box containing $s$ in the insertable position (which always exists) in row $1$ of $T$.
\item If $w_k=\bar{s}\in\{\bar{1},\bar{2},\ldots,\bar{n}\}$, add a box containing $s$ in the insertable position in either the first column to the left or the first column to the right of the wall of $T$, whichever one has the highest row number. (Notice there there is always an insertable position in those columns.)
\item Set $m$ to be the number of the row where $s$ was inserted.
\item If $s$ was added to an empty position, add a box with the number $k$ to the same position in $R$. If $k=n$ stop, otherwise, go back to (2) with $k$ replaced by $k+1$.
\item Otherwise, if a number $s'$ was replaced by $s$, consider the insertable positions for $s'$ in $T$ that are in rows $\{r\in\mathscr{R}~|~r\leq m+1\}$ (there is always at least one). Add a box containing $s'$ in the position among those with the highest row number. Replace $s$ with $s'$ and return to (4).
\end{enumerate}

\begin{theorem}
The bumping algorithm and insertion algorithm are mutual inverses of each other. 
\end{theorem}

\begin{proof}
This can be easily seen by a straightforward verification and is left as an exercise to the reader.\end{proof}

\begin{example}
Consider the element $w=275\bar{6}4\bar{3}1$. We construct a pair of bitableaux using the insertion algorithm. 

$$
\begin{array}{ccccccc}

\begin{tikzpicture}[scale=0.5]
\draw[-,line width=2pt] (0,2.5) to (0,0.5);

\node[left] at (0,2) {$2$};
\draw[-,line width=2pt] (0,-2.5) to (0,-0.5);

\node[right] at (-1,-1) {$1$};
\draw[->] (1,0) -- (2,0);
\node[above] at (1.5,0) {$7$};

\end{tikzpicture}
&

\begin{tikzpicture}[scale=0.5]
\draw[-,line width=2pt] (0,2.5) to (0,0.5);

\node[left] at (0,2) {$2$};
\node[left] at (-1,2) {$7$};

\draw[-,line width=2pt] (0,-2.5) to (0,-0.5);

\node[right] at (-1,-1) {$1$};
\node[right] at (-2,-1) {$2$};

\draw[->] (1,0) -- (2,0);
\node[above] at (1.5,0) {$5$};

\end{tikzpicture}

&

\begin{tikzpicture}[scale=0.5]
\draw[-,line width=2pt] (0,2.5) to (0,0.5);

\node[left] at (0,2) {$2$};
\node[left] at (-1,2) {$5$};
\node[left] at (1,2) {$7$};

\draw[-,line width=2pt] (0,-2.5) to (0,-0.5);

\node[right] at (-1,-1) {$1$};
\node[right] at (-2,-1) {$2$};
\node[left] at (1,-1) {$3$};
\draw[->] (1,0) -- (2,0);
\node[above] at (1.5,0) {$\bar{6}$};

\end{tikzpicture}
&

\begin{tikzpicture}[scale=0.5]
\draw[-,line width=2pt] (0,2.5) to (0,0.5);

\node[left] at (0,2) {$2$};
\node[left] at (-1,2) {$5$};
\node[left] at (1,2) {$7$};
\node[left] at (0,1) {$6$};

\draw[-,line width=2pt] (0,-2.5) to (0,-0.5);

\node[right] at (-1,-1) {$1$};
\node[right] at (-2,-1) {$2$};
\node[left] at (1,-1) {$3$};
\node[right] at (-1,-2) {$4$};

\draw[->] (1,0) -- (2,0);
\node[above] at (1.5,0) {$4$};

\end{tikzpicture}

&

\begin{tikzpicture}[scale=0.5]
\draw[-,line width=2pt] (0,2.5) to (0,0.5);

\node[left] at (0,2) {$2$};
\node[left] at (-1,2) {$4$};
\node[left] at (1,2) {$5$};
\node[left] at (0,1) {$6$};
\node[left] at (-1,1) {$7$};

\draw[-,line width=2pt] (0,-2.5) to (0,-0.5);

\node[right] at (-1,-1) {$1$};
\node[right] at (-2,-1) {$2$};
\node[left] at (1,-1) {$3$};
\node[right] at (-1,-2) {$4$};
\node[right] at (-2,-2) {$5$};

\draw[->] (1,0) -- (2,0);
\node[above] at (1.5,0) {$\bar{3}$};

\end{tikzpicture}

&

\begin{tikzpicture}[scale=0.5]
\draw[-,line width=2pt] (0,2.5) to (0,0.5);

\node[left] at (0,2) {$2$};
\node[left] at (-1,2) {$4$};
\node[left] at (1,2) {$5$};
\node[left] at (0,1) {$3$};
\node[left] at (-1,1) {$7$};
\node[left] at (1,1) {$6$};

\draw[-,line width=2pt] (0,-2.5) to (0,-0.5);

\node[right] at (-1,-1) {$1$};
\node[right] at (-2,-1) {$2$};
\node[left] at (1,-1) {$3$};
\node[right] at (-1,-2) {$4$};
\node[right] at (-2,-2) {$5$};
\node[left] at (1,-2) {$6$};

\draw[->] (1,0) -- (2,0);
\node[above] at (1.5,0) {$1$};

\end{tikzpicture}

&

\begin{tikzpicture}[scale=0.5]
\draw[-,line width=2pt] (0,2.5) to (0,0.5);

\node[left] at (0,2) {$1$};
\node[left] at (-1,2) {$4$};
\node[left] at (1,2) {$2$};
\node[left] at (0,1) {$3$};
\node[left] at (-1,1) {$5$};
\node[left] at (1,1) {$6$};
\node[left] at (2,2) {$7$};
\draw[-,line width=2pt] (0,-2.5) to (0,-0.5);

\node[right] at (-1,-1) {$1$};
\node[right] at (-2,-1) {$2$};
\node[left] at (1,-1) {$3$};
\node[right] at (-1,-2) {$4$};
\node[right] at (-2,-2) {$5$};
\node[left] at (1,-2) {$6$};
\node[left] at (2,-1) {$7$};

\end{tikzpicture}
\end{array}
$$
Therefore  $$w=275\bar{6}4\bar{3}1 \mapsto \left(\left(\young(41,53); \young(27,6) \right), \left(\young(21,43); \young(37,6) \right)\right).$$
The following diagram explains the last step of the algorithm,  where $1$ is inserted into the tableau $\left(\young(42,73); \young(5,6) \right)$.  

$$
\begin{array}{clllll}

\begin{tikzpicture}[scale=0.5]
\draw[-,line width=2pt] (0,2.5) to (0,0.5);

\node at (-3,4) {$1$};
\draw [->, color=blue] (-3,3.5) -- (-2,2.5);
\node[left] at (0,2) {$2$};
\node[left] at (-1,2) {$4$};
\node[left] at (1,2) {$5$};
\node[left] at (0,1) {$3$};
\node[left] at (-1,1) {$7$};
\node[left] at (1,1) {$6$};
\node[left] at (-2,1.5) {$\widetilde{T}_6=$};

\node at (2,1.5) {$\leadsto$};
\end{tikzpicture}

&
\begin{tikzpicture}[scale=0.5]
\draw[-,line width=2pt] (0,2.5) to (0,0.5);

\node at (-1.5,4) {$1$};
\draw [->, color=blue] (-1,3.5) -- (-0.5,2);

\node at (1,1.5) {$\leadsto$};

\end{tikzpicture}

&

\begin{tikzpicture}[scale=0.5]
\draw[-,line width=2pt] (0,2.5) to (0,0.5);

\node at (-2,4)(2) {$2$};

\node[left] at (0,2) {$1$};
\node at (1,2) (E){};
\draw (2) edge[ color=blue, ->, out=0, in=90] (E);

\node at (1,1.5) {$\leadsto$};
\end{tikzpicture}
&

\begin{tikzpicture}[scale=0.5]
\draw[-,line width=2pt] (0,2.5) to (0,0.5);

\node at (-3,4) (5) {$5$};
\node at (-1,1) (E) {};

\draw (5) edge[ color=blue, ->, out=270, in=180] (E);

\node[left] at (0,2) {$1$};
\node[left] at (-1,2) {$4$};
\node[left] at (1,2) {$2$};
\node[left] at (0,1) {$3$};
\node at (2,1.5) {$\leadsto$};

\end{tikzpicture} 

&

\begin{tikzpicture}[scale=0.5]
\draw[-,line width=2pt] (0,2.5) to (0,0.5);

\node at (-3,4) {$7$};
\node at (1,2) (E) {};
\draw (5) edge[ color=blue, ->, out=0, in=90] (E);

\node[left] at (0,2) {$1$};
\node[left] at (-1,2) {$4$};
\node[left] at (1,2) {$2$};
\node[left] at (0,1) {$3$};
\node[left] at (-1,1) {$5$};
\node[left] at (1,1) {$6$};
\node at (2,1.5) {$\leadsto$};

\end{tikzpicture} 

&

\begin{tikzpicture}[scale=0.5]
\draw[-,line width=2pt] (0,2.5) to (0,0.5);

\node[left] at (0,2) {$1$};
\node[left] at (-1,2) {$4$};
\node[left] at (1,2) {$2$};
\node[left] at (0,1) {$3$};
\node[left] at (-1,1) {$5$};
\node[left] at (1,1) {$6$};
\node[left] at (2,2) {$7$};
\end{tikzpicture}

\end{array}
 $$
Notice that here the numbers $2$ and $5$ both get bumped to the next row (from $m$ to $m+1$ in the notation of the algorithm) but the number $7$ does not have an insertable position in rows $m+1$ nor $m$, so it actually gets bumped `up' to row $m-1$, according to step (6) of the algorithm.
\end{example}

From \cite[Section 6]{NRS16} it was shown that the {\it exotic Steinberg variety} 
$$ 
\mf{Z}:=\widetilde{\fN}\times_{\fN}\widetilde{\fN}:=\{(F_\bullet,G_\bullet,(v,x))\in \cF(V)\times\cF(V)\times \fN~|~F_\bullet,G_\bullet\in\mathcal{C}_{(v,x)}\},
$$
has its irreducible components parametrised in two ways: one way by elements of the Weyl group $W(C_n)$ and the other by irreducible components of the Exotic Springer fibres, or in other words, by pairs of standard Young bitableaux. This gives rise to a bijection $$W(C_n) \longleftrightarrow  \coprod_{(\mu,\nu) \in \cQ_n}  \cT(\mu,\nu) \times \cT(\mu,\nu)$$  defined geometrically, as in \cite[Cor. 7.2]{NRS16}. The algorithms described in this section compute this geometric bijection as stated in the following theorem.

\begin{theorem}[Main Theorem] \label{theorem:main}
Let $F_\bullet, G_\bullet\in\cC^\circ_{(v,x)}$ be generic flags, with $\Phi(F_\bullet)=T$ and $\Phi(G_\bullet)=R$. Then $w(F_\bullet,G_\bullet)=\RS(T,R)$.
\end{theorem}
The proof of the theorem will be given in Section \ref{section:algorithmconstruction}.

\section{Intersecting Generic Hyperplanes} \label{section:generichyperplanes}
Through out this section $(\mu',\nu')$ will be a bipartition obtained from $(\mu,\nu)$ by decreasing either $\mu_m$ or $\nu_m$ by $1$ and $W$ and $X$ will be two generic points in $\cB_{(\mu',\nu')}^{(\mu,\nu)}$. We will further assume that the partitions $(\mu,\nu)$ and $(\mu',\nu')$ are such that the generic points $X$ and $W$ automatically satisfy $X \subset W^{\perp}$. 

In this section, we describe how $\eType(v+Y, x_{|Y^{\perp}/Y})$ is obtained from $\eType(v+W, x_{|W^{\perp}/W})=(\mu',\nu')$ in all cases; see the below Theorem. This is analogous to Lemma 3.2 in \cite{Ste}, and is the key step that allows us to describe the steps in the reverse bumping algorithm; it will be used in Lemma \ref{lemma:rs} below.

\begin{theorem} \label{theorem:etypeY}
Suppose $W$ and $X$ are two generic points in $\cB_{(\mu',\nu')}^{(\mu,\nu)}$ where $(\mu',\nu')$ is obtained from $(\mu,\nu)$ by decreasing one part of $\mu$ or $\nu$ by $1$ and that $(\mu,\nu)$ is a bipartition such that $X \subset W^{\perp}$. Put $Y=X+W$. 
\begin{enumerate}
\item Suppose that $\nu'=\nu$ and $\mu_m'=\mu_m-1$. If $m=1$, then $X=W$. Otherwise, $\eType(v+Y,x_{|Y^{\perp}/Y})$ is obtained from $(\mu',\nu')$ by decreasing:
\begin{enumerate}[(a)]
\item $\mu_m'$ by 1  if $\mu_m-1>\mu_{m+1}$ and either $\nu_{m-1}=\nu_{m}\neq 0$, or $\nu_{m-1}=0$; or
\item $\nu_{\max \Delta_m}$ by 1 if $\mu_m-1=\mu_{m+1}$, $\nu_{m-1}=\nu_{m}\neq 0$ and either $\max \Gamma_{m+1} > \max \Delta_m$, or $\mu_m=1$ (and so $\mu_{m+1}=0$); or
\item $\mu_{\max \Gamma_{m+1}}$ by 1 if  $\mu_m-1=\mu_{m+1} \neq 0$ and either $\nu_{m-1}=\nu_{m}\neq 0$ and $\max \Gamma_{m+1} \leq \max \Delta_m$, or $\nu_{m-1}=0$; or
\item $\nu_{m-1}$ by 1 if $\nu_{m-1}>\nu_{m}$. 
\end{enumerate}

\item Suppose that $\mu' =\mu$ and $\nu_m'=\nu_m-1$. Then $\eType(v+Y,x_{|Y^{\perp}/Y})$ is obtained from $(\mu',\nu')$ by decreasing:
\begin{enumerate}[(a)]
\item $\mu_m$ by 1 if $\max \Gamma_m=\max \Delta_m$; or
\item $\nu_m'$ by $1$ if $\nu_m-1>\nu_{m+1}$ and either $\max \Gamma_m > \max \Delta_m$, or $\mu_m=0$; or
\item $\nu_{\max \Delta_{m+1}}$ by 1 if $\nu_m-1=\nu_{m+1}\neq 0$ and either $\max \Gamma_m > \max \Delta_{m+1}$, or $\mu_m=0$; or
\item $\mu_{\max \Gamma_{m}}$ by $1$ if $\nu_m-1=\nu_{m+1}$ and either $\max \Gamma_{m} \leq \max \Delta_{m+1}$, or $\nu_{m}=1$  (and so $\nu_{m+1}=0$). 

\end{enumerate}
\end{enumerate}
\end{theorem}
\begin{remark}
In parts (1)(c) and (2)(c) of Theorem \ref{theorem:etypeY} we have stated that $\mu_m>1$ and $\nu_m>1$ for completeness even though it was implicitly assumed by requiring that $X \subseteq W^{\perp}$ for which the conditions of Proposition \ref{proposition:perpendicular1} apply.
\end{remark}

\begin{remark}\label{remark:43}
See the proof of Lemma \ref{lemma:rs} below for details on how Theorem \ref{theorem:etypeY} relates to the exotic Robinson-Schensted algorithm. Here we indicate how Theorem \ref{theorem:etypeY} corresponds to the reverse bumping algorithm of Section \ref{section:bumpingalgorithm}.

In Case 1, the number $s$ to be removed is in the $(m,1)$-th position. 
\begin{itemize}
\item In (a), $\mu_m - 1 > \mu_{m+1}$ and either $\nu_{m-1} = \nu_m \neq 0$, or $\nu_{m-1}=0$. Here the first ``available position'' is the box adjacent to $s$ (since $\nu_{m-1} = \nu_m$ or $\nu_{m-1}=0$, there are no such positions in the preceding row). Replacing that label with $s$ corresponds to decreasing $\mu'_m$ by $1$. 
\item In (b), $\mu_m - 1 = \mu_{m+1}$, $\nu_{m-1}=\nu_m$. If either $\text{max}(\Gamma_{m+1}) > \text{max}(\Delta_m)$ or $\mu_m=1$, then the first ``available position'' is at the end of row $\text{max}(\Delta_m)$. Replacing that label with $s$ corresponds to decreasing $\nu_{\text{max}(\Delta_m)}$ by $1$.  
\item In (c), $\mu_m -1 = \mu_{m+1}$, $\nu_{m-1} = \nu_m$ and either $\text{max}(\Gamma_{m+1}) \leq \text{max}(\Delta_m)$ or $\nu_{m-1}=0$. It follows first ``available position'' is at the end of row $\text{max}(\Gamma_{m+1})$. Replacing that label with $s$ corresponds to decreasing $\mu_{\text{max}(\Delta_m)}$ by $1$. 
\item In (d), $\nu_{m-1} > \nu_m$, and there is an ``available position'' in the first possible row: namely, the row corresponding to $\nu_{m-1}$. Replacing that label with $s$ corresponds to decreasing $\nu_{m-1}$ by $1$. 
\vspace{5pt}
One can similarly verify that Case 2 (a), (b), (c), (d) match up with the bumping algorithm.

\end{itemize}
\end{remark}

\subsection{Preliminaries.}For the rest of this section, we will assume that $W\neq X$. By the Travkin and Achar-Henderson criterion, to calculate  $\eType(v+Y,x_{|Y^{\perp}/Y})$, we need to know the following Jordan types: $\Type(x,Y^{\perp}/Y)$ and $\Type(x,Y^{\perp}/(\C[x]v+Y))$. Since $X \subset W^{\perp}$, we have $Y=X+W \subset X^{\perp} \cap W^{\perp}=Y^{\perp}$ and so we may regard $Y/W$ as a $1$-dimensional subspace of $\ker(x_{|W^{\perp}/W}) \cap (\C[x]v+W)^{\perp}/W \subset W^{\perp}/W$. The following two lemmas will be useful in this regard. 
 
 \begin{lemma} \label{lemma:maxk2}
 Let $\sigma$ be the Jordan type of $x$ restricted to the space $W^{\perp}/(\C[x]v+W)$. Then the Jordan type $\sigma'$ of the induced nilpotent $x$ on $W^{\perp}/(\C[x]v+X+W)$ is determined by the maximal $k$ such that $$ X \subseteq x^{k-1}(W^{\perp}) + \C[x]v + W. $$ If $X\subset\C[x]v+W$, then $\sigma'=\sigma$ (and there is no maximal $k$), otherwise we have that $\sigma'$ is obtained by removing the last box at the bottom of the $k$-th column of $\sigma$. 
 \end{lemma}
 \begin{proof}
 Since $W^{\perp}/(\C[x]v+X+W)$ is the quotient of the space $W^{\perp}/(\C[x]v+W)$ by the one-dimensional subspace $(\C[x]v+X+W)/(\C[x]v+W)$, we know by \cite[page 1]{Spa} that $\Type(x,W^{\perp}/(\C[x]v+X+W))$ is given by the maximal $k$ such that $$ (\C[x]v+X+W)/(\C[x]v+W) \subseteq x^{k-1}(W^{\perp}/(\C[x]v+W)).$$ Translating this back to a condition on $X$, we find that 
\begin{eqnarray*}
\C[x]v+X+W 	&\subseteq&	 x^{k-1}(W^{\perp}) + \C[x]v+W  \quad \text{and so} \\ 
X   			&\subseteq&	 x^{k-1}(W^{\perp}) + \C[x]v+W 
\end{eqnarray*} 
as required. \end{proof} 
\begin{lemma} \label{lemma:maxl2}
Let $\rho$ be the Jordan type of $x$ restricted to the space $W^{\perp}/(\C[x]v+X+W)$. Then the Jordan type $\rho'$ of the induced nilpotent $x$ on $Y^{\perp}/(\C[x]v+X+W)=Y^{\perp}/(\C[x]v+Y)$ is determined by the maximal $l$ such that 
$$
X^{\perp} \supseteq (x^{l-1})^{-1}(\C[x]v+W+X)\cap W^{\perp}.
$$  
Then $\rho'$ is obtained by deleting the last box of the $l$-th column of $\rho$
 \end{lemma} 

\begin{proof}
Since $Y^{\perp}/(\C[x]v+Y)$ is an $x$-stable hyperplane in $W^{\perp}/(\C[x]v+Y)$, we know by \cite[Lemma 1.4]{leeuwen2} that $\Type(x, Y^{\perp}/(\C[x]v+Y)$ is given by the maximal $l$ such that 
$$
Y^{\perp}/(\C[x]v+Y) \supseteq \im(x_{|W^{\perp}/(\C[x]v+X+W)})+\ker(x_{|W^{\perp}/(\C[x]v+X+W)}^{l-1}).
$$
Now 
$$ \im(x_{|W^{\perp}/(\C[x]v+X+W)})=\im(x_{|W^{\perp}})+\C[x]v+X+W=x(W^{\perp})+\C[x]v+X+W,$$ 
and 
$$\ker(x_{|W^{\perp}/(\C[x]v+X+W)}^{l-1})= (x^{l-1})^{-1}(\C[x]v+X+W)\cap W^{\perp}.$$
Thus we require the maximal $l$ such that 
\begin{equation}\label{eq:maxl}
Y^{\perp} \supseteq x(W^{\perp})+ (x^{l-1})^{-1}(\C[x]v+X+W)\cap W^{\perp} + \C[x]v+X+W.
\end{equation}
Now since $X \subset W^{\perp} \cap \ker(x)$, we have $X \subset \ker(x_{|W^{\perp}})$ and so $X^{\perp} \supseteq x(W^{\perp})$. Also, since $Y=X+W \subseteq (\C[x]v)^{\perp}$, we have that $Y^\perp\supseteq \C[x]v$, so \eqref{eq:maxl} simplifies to the maximal $l$ such that 
$$
Y^{\perp} \supseteq (x^{l-1})^{-1}(\C[x]v+X+W)\cap W^{\perp}.
$$
Translating this back to a condition on $X$ and $W$, we require the maximal $l$ such that 
$$
X^{\perp} \cap W^{\perp} \supseteq  (x^{l-1})^{-1}(\C[x]v+W+X)\cap W^{\perp}
$$ 
Now the condition that $W^{\perp} \supseteq  (x^{l-1})^{-1}(\C[x]v+W+X) \cap W^{\perp}$ is redundant and so we require 
$$ X^{\perp} \supseteq  (x^{l-1})^{-1}(\C[x]v+W+X)\cap W^{\perp}$$ and the proof is complete.\end{proof}


Now we know that $\Type(x,Y^{\perp}/(\C[x]v+Y))$ is obtained from $\Type(x,W^{\perp}/(\C[x]v+W))$ by decreasing two parts, and we will see below that these are two consecutive parts, which completely determines $\eType(v+Y,x_{|Y^{\perp}/Y})$. \vspace{5pt}


For what follows, we will always have $W=\C\{w\}$ and $X=\C\mf\{\mathtt{x}\}$ where
$$
w=\sum_{i=1}^{m} \alpha_{i}v_{i,1}+\beta_{i}v_{i,\lambda_i}^{*} \quad \text{and} \quad \mathtt{x}=\sum_{i=1}^{m} \gamma_{i}v_{i,1}+\delta_{i}v_{i,\lambda_i}^{*},
$$ 
and $m$ is the index where either $\mu$ or $\nu$ is decreased by $1$ to obtain $\eType(v+W, x_{|W^{\perp}/W})=\eType(v+X,x_{|X^{\perp}/X})$. 

\begin{proposition} \label{proposition:maxk2prim}
The maximal $k_2$ such that $ X \subseteq x^{k_2-1}(W^{\perp})+W,$ is 
$$k_2=
\begin{cases} 
\lambda_{m-1} 	&\text{if} \quad \mu' \preceq \mu \, \, \text{and} \, \, \nu_{m-1} > \nu_m, \\ 
\lambda_m-1 	&\text{otherwise.}
\end{cases}
$$
\end{proposition}

\begin{proof}
Since $X$ and $W$ are generically chosen in $\cB_{(\mu',\nu')}^{(\mu,\nu)}$, it follows that $ X \subseteq x^{\lambda_m-1}(V)$ but not  $x^{\lambda_m}(V)$. Therefore since $x(V) \subset W^{\perp}$, it immediately follows that $X \subseteq x^{\lambda_m-2}(W^{\perp})$, so the maximal $k_2 \geq \lambda_m-1$. \vspace{5pt}

Consider the set $\mathcal{S}:=\{v_{i,1}, v_{i,\lambda_{i}}^{*} \, | \, i \in \Lambda_m \}$. We claim that $\cS \subset x^{\lambda_m-1}(W^{\perp})$ if and only if $\nu_{m-1} > \nu_m$. It is clear that the set of vectors $\cS':=\{v_{i,\lambda_i}, v_{i, 1}^{*}  \, | \, i \in \Lambda_m\}$ maps onto $\cS$ via $x^{\lambda_m-1}$. Moreover it is easily seen that $\cS' \not\subseteq W^{\perp}$ whenever $|\Lambda_m| \geq 2$ since $W$ is generically chosen implying that $\alpha_i$'s and the $\beta_i$'s are non-zero as $i$ runs through $\Lambda_m$. Indeed since $W$ and $X$ are chosen generically, we have $\cS' \subseteq W^{\perp}$ if and only if $|\Lambda_m|=1$ and either $\alpha_m=0$ or $\beta_m=0$ in the expression for $w$ given in \eqref{eqn:genericw}. But for a generic $W$ in $\cB_{(\mu',\nu')}^{(\mu,\nu)}$, since $\mu'_m = \mu_m-1$, we can only assume that $\beta_m=0$ precisely when $\nu_{m-1} > \nu_m$ by Remark \ref{corollary:genericw} (which also implies that $\delta_m=0$ in the expression for $\ttx$). This establishes the claim. \vspace{5pt}

Now we can quickly see that if $\nu_{m-1}=\nu_m$ (which includes the case $\nu_{m-1}=\nu_m=0$), then $\cS$ and hence $X$ is not contained in $x^{\lambda_m-1}(W^{\perp})$. Moreover since $|\cS|>1$ and $X$ is generically chosen, it is impossible for $X$ to be contained in $x^{\lambda_m-1}(W^{\perp})+W$ as adding $W$ only the dimension by $1$. Therefore the maximal $k_2$ in this case is $\lambda_m-1$.  \vspace{5pt}

Now suppose that $\nu_{m-1} > \nu_m$. Then the expressions for the spanning vectors $w$ and $\ttx$ are as in \eqref{eqn:genericwx} and we claim that the set of vectors $\{v_{i,1}, v_{i,\lambda_i}^{*}  \, | \, 1 \leq i \leq m-1 \}$ are contained in $x^{\lambda_{m-1}-1}(W^{\perp})$. This is clear for all $i$ such that $\lambda_i>\lambda_{m-1}$, so we only need to check this for $i \in \Lambda_{m-1}$, that is the $i$ such that $\lambda_i=\lambda_{m-1}$. Notice that the vectors of the form $\alpha_{i}v_{i,\lambda_i}+\beta_{i}v_{i,1}^{*}$ are all contained in $W^{\perp}$
and so 
$$
x^{\lambda_{m-1}-1}(\alpha_{i}v_{i,\lambda_i}+\beta_{i}v_{i,1}^{*})=\alpha_{i}v_{i,1} \in x^{\lambda_{m-1}-1}(W^{\perp}),
$$ 
for $i \in \Lambda_{m-1}$. Therefore $v_{i,1}, v_{i,\lambda_i}^*$ are contained in $x^{\lambda_{m-1}-1}(W^{\perp})$ for $1 \leq i \leq m-1$ and so
$$
X \subseteq \C\{v_{i,1}, v_{i,\lambda_i}^{*}  \, | \, 1 \leq i \leq m-1 \} \oplus \C\{v_{m,1}\} \subseteq  x^{\lambda_{m-1}-1}(W^{\perp})+W.
$$
Since the vectors $v_{i,1}$ and $v_{i,\lambda_{i}}^{*}$ for $i \in \Lambda_{m-1}$ are not contained in $x^{\lambda_{m-1}}(V)$, they cannot be contained in $x^{\lambda_{m-1}}(W^{\perp})$ and in this case, adding the space $W$ cannot account for these missing vectors as well as $v_{m,1}$. So we have shown that the maximal $k_2$ in this case is $\lambda_{m-1}$ which completes the proof. 
\end{proof}




\begin{proposition} \label{proposition:genericl2}
Let $X$ and $W$ be two generic subspaces of $\ker(x) \cap (\C[x]v)^{\perp}$ such that $X \subseteq W^{\perp}$ and $\eType(v+X,x_{|X^{\perp}/X})=\eType(v+W,x_{|W^{\perp}/W})=(\mu',\nu').$ Then the maximal $l_2$ such that $$X^{\perp} \supseteq (x^{l_2-1})^{-1}(\C[x]v+X+W) \cap W^{\perp}$$ is $\lambda_m-1$ except when $\mu' \preceq \mu, \nu'=\nu$ and $\nu_{m-1} > \nu_m$. 
\end{proposition}
\begin{proof}
Consider the vector $w'=\sum_{i=1}^{m} \alpha_{i}v_{i,\lambda_m}+\beta_{i}v_{i,\lambda_i-\lambda_m+1}^{*}$. We have that $x^{\lambda_m-1}(w')=w$, and $w'\in W^\perp$ because $$\langle w,w' \rangle=\sum_{i \in \Lambda_m}(\alpha_{i}\beta_{i}-\beta_{i}\alpha_{i})=0.$$
but $$\langle \ttx,w' \rangle=\sum_{i \in \Lambda_m}\alpha_{i}\delta_{i}-\beta_{i}\gamma_{i},$$ which is generically non-zero except when $\mu' \preceq \mu, \nu'=\nu$ and $\nu_{m-1} > \nu_m$, which forces $|\Lambda_m|=1$ and so $\beta_m=\delta_m=0$ by Corollary \ref{corollary:genericw}. Thus when this is not the case, we have $(x^{\lambda_m-1})^{-1}(W) \not\subseteq X^{\perp}$. Furthermore, since $X \subseteq x^{\lambda_m-1}(V)$ it follows that $X^{\perp} \supseteq \ker(x^{\lambda_m-1}) \supseteq  \ker(x^{\lambda_m-2})$, and it is easily seen that $X^{\perp} \supseteq x^{\lambda_m-2}(W).$ Therefore $X^{\perp} \supseteq (x^{\lambda_m-2})^{-1}(\C[x]v+X+W)\cap W^{\perp}.$ 
\end{proof}

\begin{remark}
We will deal with the cases $\mu' \preceq \mu$, $\nu'=\nu$ and $\nu_{m-1} > \nu_m$ in Proposition \ref{proposition:D1l2}.
\end{remark}

\subsection{Box Removed from the Left}
We are now getting to the proof of Theorem \ref{theorem:etypeY}. We first consider the case where $\mu'_m=\mu_m-1$.
\subsubsection{ $\mu' \preceq \mu, \nu'=\nu$ and $\nu_{m-1}=\nu_{m}$} 

Recall that the condition $\nu_{m-1}=\nu_{m}$ is equivalent to $|\Delta_{\leq m}| \geq 2$ in the case that $\nu_{m-1}=\nu_{m} \neq 0$ and $\ell(\nu)+1<m$ in the case that $\nu_{m-1}=\nu_m=0$. Also recall from Proposition \ref{proposition:genericw} that a generic point in $\cB_{(\mu',\nu')}^{(\mu,\nu)}$ has the following conditions on the coefficients of its spanning vector; namely for a spanning vector $w$ as in \eqref{eqn:genericwx} we have: 

\begin{align*}
\sum_{i \in\Delta_{\leq m}} \beta_i=0 	\quad	&\text{whenever}  \quad  		\nu_{m-1}=\nu_{m} \neq 0; \quad \text{and} \\
\sum_{i=\ell(\nu)+1}^{m} \beta_i=0   	\quad 	&\text{whenever} \quad   		\nu_{m-1}=\nu_{m}=0. 
\end{align*}
Example \ref{example:partition1} (in the Appendix) serves as a guide to the remainder of this section. 
\begin{proposition} \label{proposition:trick}
Let $X$ and $W$ be generic points in $\cB_{(\mu',\nu')}^{(\mu,\nu)}$. Then the maximal $k_2$ such that $$X \subseteq x^{k_2-1}(W^{\perp}) + \C[x]v + W$$ is $\lambda_{m}-1$. 
\end{proposition}
 \begin{proof}
We know by Proposition \ref{proposition:maxk2prim} that $k_2 \geq \lambda_{m}-1$. We first suppose that $|\Delta_{\leq m}| \geq 2$. To show that $k_2=\lambda_m-1$ we show that the vectors $v_{i,\mu_i-\mu_m+1}$ for $i \in  \Delta_{\leq m}$ are not contained in $x^{\lambda_m-1}(W^{\perp})+\C[x]v$ which implies the vectors $v_{i,\lambda_i}^{*}$ are not contained in $x^{\lambda_m-1}(W^{\perp}) + \C[x]v + W$ for $i \in  \Delta_{\leq m}$ and so $X \not\subseteq x^{\lambda_m}(W^{\perp}) + \C[x]v + W$. \vspace{5pt}

Consider the vectors $\beta_m v_{i,\lambda_i} - \beta_i v_{m,\lambda_m}$ for $i \in   \Delta_{\leq m}$. These vectors are clearly contained in $W^{\perp}$ and so $\beta_m v_{i,\mu_i - \mu_m+1} - \beta_i v_{m,1}=x^{\lambda_m-1}(\beta_{m} v_{i,\lambda_i} - \beta_{i}v_{m,\lambda_{m}})$ is contained in $x^{\lambda_m-1}(W^{\perp})$. Now if each $v_{i,\mu_i-\mu_m+1}$ were in $x^{\lambda_m-1}(W^{\perp}) + \C[x]v,$ we would require that the vector $$v'=\sum_{i \in \Delta_{\leq m}}v_{i,\mu_i-\mu_m+1}=x^{\mu_i-\mu_m}(v)-\sum_{i=1}^{m'}v_{i, \mu_i-\mu_m+1},$$ where $m'=\min \Delta_m-1$ lie outside the span of $\beta_m v_{i,\mu_i - \mu_m+1} - \beta_i v_{m,1}$ as $i$ runs through $\Delta_{\leq m}$. However we calculate 
 \begin{eqnarray*}
 \sum_{i \in \Delta_{< m}}\beta_m v_{i,\mu_i - \mu_m+1} - \beta_i v_{m,1} 	&=& \beta_m  \sum_{i \in \Delta_{< m}}v_{i,\mu_i - \mu_m+1} - \left( \sum_{i \in \Delta_{< m}} \beta_i \right)v_{m,1} \\
 														&=& \beta_m  \sum_{i \in \Delta_{< m}}v_{i,\mu_i - \mu_m+1} + \beta_m v_{m,1} \quad \text{since} \quad \sum_{i \in \Delta_{\leq m}}\beta_i=0\\
														&=& \beta_m \sum_{i \in \Delta_{\leq m}}v_{i,\mu_i - \mu_m+1} \\
														&=& \beta_mv'.
 \end{eqnarray*}
Hence we cannot obtain the individual basis vectors $v_{i,\mu_i-\mu_m+1}$ in $x^{\lambda_m-1}(W^{\perp}) + \C[x]v$, for $i \in \Delta_{\leq m}$. In particular, we cannot obtain the vectors  $v_{i,1}$ for $i \in \Lambda_m$ by adding $\C[x]v$ which implies that we cannot obtain the vectors $v_{i,\lambda_i}^{*}$ for $i \in \Lambda_{m}$ by adding $\C[x]v+W$.  Therefore we cannot obtain $X$ as a subspace of $x^{\lambda_m-1}(W^{\perp}) + \C[x]v +W$.  \vspace{5pt}

Now suppose that $\ell(\nu)+1 < m$. Then by repeating the arguments above with this time changing the indexing set for $i$, we see the vectors $v_{i,1}$ for $\ell(\nu)+1 \leq i \leq m$ are not contained in the space $x^{\lambda_m-1}(W^{\perp})+\C[x]v$ and so the vectors  $v_{i,\lambda_i}^{*}$ are not contained in $x^{\lambda_m-1}(W^{\perp})+\C[x]v+W$. Therefore the maximal $k_2$ is $\lambda_m-1$ and we are done.\end{proof}
 
\begin{proof}[Proof of Theorem \ref{theorem:etypeY}: Cases 1(a),(b),(c)] We can now describe how exactly $\Type(x,Y^{\perp}/(\C[x]v+Y)$ is obtained from $\Type(x,W^{\perp}/(\C[x]v+W))$ when  $\mu' \preceq \mu, \nu'=\nu$ and $\nu_{m-1}=\nu_{m}$. In this case we have by Proposition  \ref{proposition:trick} and Proposition \ref{proposition:genericl2}, $k_2=l_2=\lambda_m-1$.  Let 
$$ \sigma=\Type(x,W^{\perp}/(\C[x]v+W) \quad \text{and} \quad \sigma'=\Type(x,Y^{\perp}/(\C[x]v+Y)) $$ 
We explicitly determine which parts of $\sigma$ are decreased by $1$ in order to obtain $\sigma'$. We have 
$$
\sigma=(\mu_1+\nu_1, \mu_2+\nu_1, \ldots, \mu_{m-1}+\nu_{m-1}, \mu_{m}-1+\nu_{m-1},  \mu_{m}-1+\nu_{m}, \mu_{m+1}+\nu_{m},\ldots) 
$$
 and we know that that $\sigma'$ is obtained from $\sigma$ by decreasing the last two parts of size $\lambda_m-1=\mu_m+\nu_m-1$. Since $\nu_{m-1}=\nu_m$, we know that there certainly are at least two parts of $\sigma$ have have this size, namely $\sigma_{2(m-1)}=\mu_{m}-1+\nu_{m-1}$ and $\sigma_{2m-1}=\mu_{m}-1+\nu_{m}$. \vspace{5pt}

\textbf{Case 1(a):} If $\mu_{m}-1 > \mu_{m+1}$, then  $\mu_{m}-1+\nu_{m} > \mu_{m+1}+\nu_{m}$ and so $\sigma_{2(m-1)}$ and $\sigma_{2m-1}$ are the last two consecutive parts of size $\lambda_m-1$, and so in that case we have 
$$
\sigma'=(\mu_1+\nu_1, \mu_2+\nu_1, \ldots, \mu_{m-1}+\nu_{m-1}, \mu_{m}-2+\nu_{m-1},  \mu_{m}-2+\nu_{m}, \mu_{m+1}+\nu_{m},\ldots). 
$$
Therefore $\eType(v+Y,x_{|Y^{\perp}/Y})$ is obtained from $\eType(v+W,x_{|W^{\perp}/W})$ by decreasing $\mu_{m}'=\mu_{m}-1$ by $1$ (note that $\nu_{m-1}$ and $\nu_{m}$ can both be $0$ here, we only require that they are equal). Thus we have
$$
\eType(v,x)\xrightarrow{\text{decrease $\mu_{m}$ by 1}}\eType(v+W,x_{|W^{\perp}/W})\xrightarrow{\text{decrease $\mu_{m}'$ by 1}}\eType(v+Y,x_{|Y^{\perp}/Y}).
$$
In terms of the reverse bumping algorithm, 
this corresponds to a number being removed from the left and remaining on the left and on the same row. \vspace{5pt}

\textbf{Case 1(b):} Suppose that $\mu_{m}-1=\mu_{m+1}$, and either $\max \Gamma_{m+1} > \max \Delta_m$ or $\mu_{m+1}=0$, then the last two consecutive parts of $\sigma$ of size $\lambda_m-1$ are 
$$ \sigma_{2\max \Delta_{m}-1}=\mu_{\max \Delta_{m}}+\nu_{\max \Delta_{m}} \quad \text{and} \quad \sigma_{2\max \Delta_{m}}=\mu_{\max \Delta_{m}+1}+\nu_{\max \Delta_{m}} $$
Therefore in this case $\eType(v+Y,x_{|Y^{\perp}/Y})$ is obtained from $\eType(v+W,x_{|W^{\perp}/W})$ by decreasing $\nu_{\max \Delta_m}$ by $1$. Thus we have

$$
\eType(v,x)\xrightarrow{\text{decrease $\mu_{m}$ by 1}}\eType(v+W,x_{|W^{\perp}/W})\xrightarrow{\text{decrease $\nu_{\max \Delta_m}$ by 1}}\eType(v+Y,x_{|Y^{\perp}/Y}).
$$

This corresponds to a number being removed from the left and moving across the wall and possibly moving down several rows on the right.

 \textbf{Case 1(c):} Suppose that $\mu_{m}-1=\mu_{m+1}$, and $\max \Gamma_{m+1} \leq \max \Delta_m$. Then the last two consecutive parts of $\sigma$ of size $\lambda_m -1$ are: 
 $$ \sigma_{2(\max \Gamma_{m+1}-1)}=\mu_{\max \Gamma_{m+1}}+\nu_{\max \Gamma_{m+1}-1} \quad \text{and} \quad \sigma_{2\max \Gamma_{m+1}-1}=\mu_{\max \Gamma_{m+1}}+\nu_{\max \Gamma_{m+1}}.
$$
Therefore $\eType(v+Y,x_{|Y^{\perp}/Y})$ is obtained from $\eType(v+W,x_{|W^{\perp}/W})$ by decreasing $\mu_{\max  \Gamma_{m+1}}$ by $1$. Thus we have
$$
\eType(v,x)\xrightarrow{\text{decrease $\mu_{m}$ by 1}}\eType(v+W,x_{|W^{\perp}/W})\xrightarrow{\text{decrease $\mu_{\max  \Gamma_{m+1}}$ by 1}}\eType(v+Y,x_{|Y^{\perp}/Y})
$$
This corresponds to a number being removed from the left and moving down possibly several rows but remaining on the left. \vspace{5pt}
 \end{proof}

\subsubsection{ $\mu' \preceq \mu, \nu' =\nu$ and $\nu_{m-1}>\nu_{m}$}
In this case Corollary \ref{corollary:genericw} forces us to have the coefficients of $v_{m,\lambda_m}^{*}$, namely $\beta_m$ and $\delta_m$ to be $0$, and so a generic $w$ and $\ttx$ have the form 
$$
w=\sum_{i=1}^{m-1}(\alpha_{i}v_{i,1}+\beta_{i}v_{i,\lambda_i}^{*})+\alpha_{m}v_{m,1} \quad \text{and} \quad \ttx=\sum_{i=1}^{m-1}(\gamma_{i}v_{i,1}+\delta_{i}v_{i,\lambda_i}^{*})+\gamma_{m}v_{m,1} 
$$ 
with $\alpha_m \neq 0 \neq \gamma_m$. Note it is possible for $\nu_{m}=0$ as we only require that $\nu_{m-1}>\nu_{m}$. 

Example \ref{example:5.2.2} (in the Appendix) serves as a guide to the rest of this section.
\begin{proposition} \label{proposition:D1k2}
Let $W$ and $X$ be generic points in $\cB_{(\mu',\nu')}^{(\mu,\nu)}$. Then maximal $k_2$ such that $X \subset x^{k_2-1}(W^{\perp})+\C[x]v+W$ is $\lambda_{m-1}$. 
\end{proposition}
\begin{proof}
By Proposition \ref{proposition:maxk2prim}, we know that $\lambda_{m-1}$ is the largest integer such that $X \subset x^{\lambda_{m-1}-1}(W^{\perp})+W$, so $k_2 \geq \lambda_{m-1}$. Also by the proof of Proposition \ref{proposition:maxk2prim}, we know that the set of vectors $\{ v_{i,1}, v_{i,\lambda_i}^{*} \, | \, 1 \leq i \leq m-1\} \subset x^{\lambda_{m-1}-1}(W^{\perp})$ whereas the set of vectors $\{v_{i,1}, v_{i,\lambda_i}^{*} \, | \, i \in \Lambda_{m-1}\} \cup \{v_{1,m}\} \not\subseteq x^{\lambda_m}(W^{\perp})+W$.  It once again follows that the set of vectors $\{v_{i,1}, v_{i,\lambda_i}^{*} \, | \, i \in \Lambda_{m-1} \} \cup \{v_{1,m}\}$ are not contained in $x^{\lambda_m}(W^{\perp})+\C[x]v+W$ since we cannot obtain the individual vectors $\{v_{i,1}\}$ for $ i \in \Lambda_{m-1}$ and the vector $\{v_{1,m}\}$ by adding the subspace $\C[x]v$ and so we cannot then obtain the individual vectors $v_{i,\lambda_i}^{*}$ for $ i \in \Lambda_{m-1}$ either. Therefore $X \not\subseteq x^{\lambda_m}(W^{\perp})+W$ and $k_2=\lambda_{m-1}$ completing the proof. 
\end{proof}

\begin{proposition} \label{proposition:D1l2}
Let $W$ and $X$ be generic points in $\cB_{(\mu',\nu')}^{(\mu,\nu)}$. Then the maximal $l_2$ such that $$X^{\perp} \supseteq (x^{l_2-1})^{-1}(\C[x]v+X+W) \cap W^{\perp}$$ is $\mu_m+\nu_{m-1}-1$. 
\end{proposition}

\begin{proof}
 Consider the vector 
$$
u:=x^{\mu_m-1}(\alpha_m v)-w=\sum_{i=1}^{m-1}\alpha_m v_{i,\mu_i-\mu_m+1}-(\alpha_{i}v_{i,1} + \beta_{i}v_{i,\lambda_i}^{*}) \in \C[x]v+X+W.$$
 Observe that this vector is not supported on $v_{m,1}$.
We define 
$$u':= 
\sum_{i=1}^{m-1}\alpha_mv_{i,\mu_i+\nu_{m-1}}-(\alpha_{i}v_{i,\mu_m+\nu_{m-1}}+\beta_{i}v_{i,\lambda_i -\mu_m-\nu_{m-1}+1}^{*})
$$

and observe that $x^{\mu_m+\nu_{m-1}-1}(u')=u$. 
Now, we let
$$u'':=u'-\frac{\langle w,u' \rangle}{\alpha_m} v_{m,1}^{*}$$ 
such that $x^{\mu_m+\nu_{m-1}-1}(u'')=u$ and
$$ \langle w, u''\rangle=\langle w,u'\rangle-\frac{\langle w,u'\rangle}{\alpha_m} \langle w, v_{m,1}^*\rangle=\langle w,u'\rangle \left(1-\frac{\alpha_m}{\alpha_m}\right)=0$$
so $u'' \in W^{\perp}.$
Therefore $u'' \in (x^{\mu_m+\nu_{m-1}-1})^{-1}(\C[x]v+X+W)$. \vspace{5pt}

However, an easy calculation shows that 
$$
\langle \ttx,u'' \rangle =
\begin{cases}
\sum\limits_{i\in \Lambda_{m-1}}\alpha_m \delta_{i}-\gamma_m\beta_{i} 									&\text{if} \quad \mu_{m-1}>\mu_m, \\
\sum\limits_{i\in \Gamma_m \cap \Delta_{m-1}}     (\alpha_i-\alpha_m)\delta_i+(\gamma_m-\gamma_i)\beta_i		&\text{if} \quad \mu_{m-1}=\mu_m,
\end{cases}
$$
which is generically non-zero. Hence $u'' \not\in X^{\perp}$ and so $X^{\perp} \not\supseteq  (x^{\mu_m+\nu_{m-1}-1})^{-1}(\C[x]v+X+W) \cap W^{\perp}.$ \vspace{5pt}

It remains to show that $X^{\perp} \supseteq  (x^{\mu_m+\nu_{m-1}-2})^{-1}(\C[x]v+X+W) \cap W^{\perp}$. Since $\mu_m+\nu_{m-1}-2 \leq \mu_{m-1}+\nu_{m-1}-2$, the only vector in $\ker(x^{\mu_{m-1}+\nu_{m-1}-2})$ with non-zero inner product with $\ttx$ is $v_{m,1}^{*}$. However this vector is not contained in $W^{\perp}$ and so  $X^{\perp} \supseteq \ker(x^{\mu_{m-1}+\nu_{m-1}-2}) \cap W^{\perp}$.  Finally let $T:=\C(\{ v_{i,\lambda_i}, v_{i,1}^{*} \, | \, 1 \leq i \leq m-1\}\cup\{v_{m,1}^*\})$; this is the span of the basis elements that have non-zero inner product with $w$ and $\ttx$.
 Then any vector in $V$ that maps onto $u$ 
 under $x^{\mu_{m-1}+\nu_{m-1}-2}$ can easily be seen not to be supported on $T \cap W^{\perp}$ and so we have shown that $X^{\perp} \supseteq  (x^{\mu_m+\nu_{m-1}-2})^{-1}(\C[x]v+X+W) \cap W^{\perp}.$ Therefore $l_2=\mu_m+\nu_{m-1}-1$.
\end{proof}

\begin{proof}[Proof of Theorem \ref{theorem:etypeY}: Case 1(d)] We now explicitly describe which parts of $\sigma$ are removed to obtain $\sigma'$. Since $\nu_{m-1}>\nu_m$ the last two consecutive parts of $\sigma$ of sizes $k_2=\lambda_{m-1}=\mu_{m-1}+\nu_{m-1}$ and $l_2=\mu_{m}+\nu_{m-1}-1$ are $\sigma_{2(m-1)-1}$ and $\sigma_{2(m-1)}$. Therefore  $\sigma'$ is obtained from $\sigma$ by decreasing $\nu_{m-1}$ by $1$; explicitly we have
$$
\sigma'=(\mu_1+\nu_1, \mu_2+\nu_1, \ldots, \mu_{m-1}+\nu_{m-1}-1, \mu_{m}-2+\nu_{m-1},  \mu_{m}-1+\nu_{m}, \mu_{m+1}+\nu_{m},\ldots)
$$
and so
$$
\eType(v,x)\xrightarrow{\text{decrease $\mu_{m}$ by 1}}\eType(v+W,x_{|W^{\perp}/W})\xrightarrow{\text{decrease $\nu_{m-1}$ by 1}}\eType(v+Y,x_{|Y^{\perp}/Y}).
$$ \end{proof}

\subsection{Box Removed from the Right}

\subsubsection{ $\mu'= \mu, \nu' \preceq \nu$, $\max\Gamma_m = \max\Delta_m$ and $\nu_{m-1}=\nu_{m}\neq 0$}

Recall the condition $\nu_{m-1}=\nu_{m} \neq 0$ is equivalent to $|\Delta_{\leq m}| \geq 2$. Throughout this section, the maximal $l_2$ is $\lambda_m-1$ as given in Proposition \ref{proposition:genericl2}. 

\begin{proposition} \label{proposition:D2k2} 
Suppose $\eType(v+W, x_{|W^{\perp}/W})$ is obtained from $(\mu,\nu)$ by decreasing $\nu_m$ by $1$. Then the maximal $k$ such that $X \subseteq x^{k-1}(W^{\perp}) + \C[x]v + W$ is $\lambda_m$. 
\end{proposition}

\begin{proof}
In this case we are required to have $\sum_{i \in \Delta_{\leq m}} \beta_i \neq 0$ for otherwise  $\eType(v+W, x_{|W^{\perp}/W})$ would be obtained from $(\mu,\nu)$ by decreasing $\mu_m$ by $1$. It is clear that the vectors $v_{i,1}, v_{i,\lambda_{i}^{*}}$ are contained in $x^{\lambda_m-1}(W^{\perp})$ for all $i \not\in \Delta_m$ and that the vectors $v_{i,1}, v_{i,\lambda_{i}^{*}}$ are not contained in $x^{\lambda_m-1}(W^{\perp})$ for all $i \in \Delta_m \setminus \Gamma_m$. We show that $v_{i,1}$ are contained in $x^{\lambda_m-1}(W^{\perp})+\C[x]v$ for all $i \in \Delta_m$. \vspace{5pt}

As in Proposition \ref{proposition:trick}, we have $\beta_m v_{i,\lambda_i} - \beta_i v_{m,\lambda_m} \in W^{\perp}$ for $i \in \Delta_m$, and so 
$$
\beta_m v_{i,\mu_i - \mu_m+1} - \beta_i v_{m,1}=x^{\lambda_m-1}(\beta_{m} v_{i,\lambda_i} - \beta_{i}v_{m,\lambda_{m}})
$$ 
is contained in $x^{\lambda_m-1}(W^{\perp})$ for all $i \in \Delta_m$. In this case however, we must have that 
$$
v'=x^{\mu_m-1}(v)-\sum_{i=1}^{\min\Delta_m-1}v_{i,\mu_i-\mu_m+1}=\sum_{i \in \Delta_{\leq m}}v_{i,\mu_i-\mu_m+1}
$$ lies outside the span of the vectors $\beta_m v_{i,\mu_i - \mu_m+1} - \beta_i v_{m,1}$ for $i \in \Delta_{\leq m}$ since  $\sum_{i \in \Delta_{\leq m}} \beta_i \neq 0$. Therefore we have $\C\{v_{i,1} \, | \, i \in \Delta_m \}=\C\{\beta_m v_{i,\mu_i - \mu_m+1} - \beta_i v_{m,1}  \, | \, i \in \Delta_m\} \oplus \C v'$ and so we have 
$$U:=\C\{v_{i,1}, v_{j,\lambda_j}^{*} \, | \, 1 \leq i \leq m, \, 1 \leq j \leq \min \Delta_m-1 \} \subseteq x^{\lambda_{m}-1}(W^{\perp}) + \C[x]v.
$$ 
We now show that the vectors $v_{i,\lambda_i}^{*}$ for $i \in \Delta_m$ are contained in $x^{\lambda_{m}-1}(W^{\perp}) + \C[x]v+W$. We have that $\alpha_m v_{i,1}^{*} - \alpha_i v_{m,1}^{*} \in W^{\perp}$ and so 
$$
\alpha_m v_{i,\lambda_i-\lambda_m+1}^{*} - \alpha_i v_{m,\lambda_m}^{*}=x^{\lambda_{m}-1}(\alpha_m v_{i,1}^{*} - \alpha_i v_{m,1}^{*}) \in x^{\lambda_{m}-1}(W^{\perp}).
$$
Therefore we have that 
$$
U':=\C\{v_{k,1}, v_{j,\lambda_j}^{*}, \alpha_m v_{i,\lambda_i-\lambda_m+1}^{*} - \alpha_i v_{m,\lambda_m}^{*} \, | \, 1 \leq k \leq m, \, 1 \leq j \leq \min \Delta_m-1, \, i \in \Delta_m \}
$$
is contained in  $x^{\lambda_{m}-1}(W^{\perp}) + \C[x]v$. Hence $U'+W \subseteq x^{\lambda_{m}-1}(W^{\perp}) + \C[x]v + W$. Now since $W$ is generically chosen in $\cB_{(\mu',\nu')}^{(\mu,\nu)}$, we have that $w':=\sum_{i \in \Delta_m}\beta_iv_{i,\lambda_i}^{*}$, the part of $w$ that is not supported on $U$, is not contained in $\C\{\alpha_m v_{i,\lambda_i-\lambda_m+1}^{*} - \alpha_i v_{m,\lambda_m}^{*} \, | \, i \in \Delta_m\}$. Therefore by counting dimensions we have 
$$
\C\{\alpha_m v_{i,\lambda_i-\lambda_m+1}^{*} - \alpha_i v_{m,\lambda_m}^{*} \, | \, i \in \Delta_m\} \oplus \C w' = \C\{v_{i,\lambda_i}^{*}\, | \, i \in \Delta_m\}
$$ 
and so we have $ X \subseteq \C\{v_{i,1}, v_{i,\lambda_i}^{*}\, | \, 1 \leq i \leq m\}=U'+W \subseteq x^{\lambda_{m}-1}(W^{\perp}) + \C[x]v+W=x^{\lambda_{m}-1}(W^{\perp}) + \C[x]v$ since $W \subset x^{\lambda_m-1}(W^{\perp})$. \vspace{5pt}

By repeating the argument at the end of the proof in Proposition \ref{proposition:trick}, this shows that $k_2=\lambda_m$ is maximal and so we are done. \end{proof}

\begin{proof}[Proof of Theorem \ref{theorem:etypeY}: Case 2(a), $\nu_{m-1} = \nu_m \neq 0$] We now explicitly determine $\eType(v+Y, x_{|Y^{\perp}/Y})$ by determining $\sigma'=\Type(x,Y^{\perp}/(\C[x]v+Y))$. We have 
\begin{eqnarray*}
\sigma 	&=& \Type(x,W^{\perp}/(\C[x]v+W)) \\
		&=& (\mu_1+\nu_1,\mu_2+\nu_1,\ldots, \mu_{m}+\nu_{m-1},\mu_{m}+\nu_{m}-1, \mu_{m+1}+\nu_{m}-1,\mu_{m+1}+\nu_{m+1},\ldots)
\end{eqnarray*}
Since $\nu_{m-1}=\nu_{m}\neq 0$ we have $\nu_{m-1} >\nu_{m}-1$. Therefore the parts of $\sigma$ of size $k_2$ and $l_2$ are: 
$$
\sigma_{2(m-1)}=\mu_{m}+\nu_{m-1}=\lambda_m \quad \text{and} \quad \sigma_{2m-1}=\mu_{m}+\nu_{m}-1=\lambda_{m}-1.
$$
It follows that $\sigma'$ is obtained from $\sigma$ by decreasing $\mu_m$ by $1$. Thus we have 
$$
\eType(v,x)\xrightarrow{\text{decrease $\nu_{m}$ by 1}}\eType(v+W,x_{|W^{\perp}/W})\xrightarrow{\text{decrease $\mu_{m}$ by 1}}\eType(v+Y,x_{|Y^{\perp}/Y}).
$$ \end{proof}

\subsubsection{$\mu'= \mu, \nu' \preceq \nu$, $\max\Gamma_m = \max\Delta_m$ and $\nu_{m-1}>\nu_{m}\neq 0$}

 \begin{proposition} \label{proposition:k2delta1}
 Let $W$ and $X$ be as above, then the maximal $k_2$ such that $X \subseteq x^{k-1}(W^{\perp}) + \C[x]v + W$ is  $$k_2=\begin{cases} \mu_{m}+\nu_{m-1} &\text{if} \quad  m> 1, \\ \infty &\text{if} \quad m=1. \end{cases}$$
 \end{proposition}
 \begin{proof}
We first deal with the case $m=1$. In this case, the condition $\nu_{m-1}>\nu_{m}\neq 0$ should be ignored.
Then our vectors $w$ and $\ttx$ have the form $$w=\alpha v_{11} + \beta v_{1,\lambda_1}^{*} \quad \text{and} \quad \ttx = \gamma v_{11} + \delta v_{1,\lambda_1}^{*}.$$ Moreover, since  
 $\max\Gamma_m = \max\Delta_m$, we have $v_{11}=x^{\mu_1-1}(v)$ and so it follows that $\ttx \subset \{x^{\mu_1-1}(v), w\} \subset \C[x]v + W.$ Therefore we set $k_2=\infty$. \vspace{5pt}

Now suppose that $m >1$. Clearly the vectors $v_{i,1}$ and $v_{i,\lambda_i}^{*}$ are all contained in $x^{\mu_m+\nu_{m-1}-1}(W^{\perp})$ for $1 \leq i \leq \min \Lambda_{m-1}-1$. We now show that the vectors $v_{i,\mu_i+\mu_m+1}, v_{i, \lambda_i}^{*}$ are contained in $x^{\mu_m+\nu_{m-1}-1}(W^{\perp})$ for $i \in \Delta_{m-1}$. Using the vectors $v_{m,1}$ and $v_{m,\lambda_m}^{*}$ as pivots we see that the vectors $\beta_{m}v_{i,\lambda_i}-\beta_{i}v_{m,\lambda_m}$ and $\alpha_{m}v_{i,1}^{*}-\alpha_{i}v_{m,1}^{*}$ are contained in $W^{\perp}$ for $i \in \Delta_{m-1}$. Hence 
 $$
 \beta_{m}v_{i,\mu_i-\mu_m+1}=x^{\mu_m+\nu_{m-1}-1}(\beta_{m}v_{i,\lambda_i}-\beta_{i}v_{m,\lambda_m}) \in x^{\mu_m+\nu_{m-1}-1}(W^{\perp})
 $$
 and 
 $$
\alpha_{m}v_{i,\lambda_i}^{*}=x^{\mu_m+\nu_{m-1}-1}(\alpha_{m}v_{i,1}^{*}-\alpha_{i}v_{m,1}^{*})\in x^{\mu_m+\nu_{m-1}-1}(W^{\perp})
$$ for $i \in \Delta_{m-1}.$ Now $x^{\mu_m-1}(v)=\sum_{i=1}^{m}v_{i,\mu_i-\mu_m+1}$ and since $v_{i,\mu_i-\mu_m+1} \in x^{\mu_m+\nu_{m-1}-1}(W^{\perp})$ for all $1 \leq i \leq m-1$, we conclude that $v_{m,1} \in x^{\mu_m+\nu_{m-1}-1}(W^{\perp}) + \C[x]v$ as well. Hence 
\begin{equation} \label{eqn:addW}
\C\{v_{i,1},v_{i, \lambda_{i}}^{*} \, | \, 1 \leq i \leq m-1\} \oplus \C\{v_{m,1}\} \subset x^{\mu_m+\nu_{m-1}-1}(W^{\perp}) + \C[x]v.
\end{equation}
Now adding $W$ to each side of \eqref{eqn:addW}, we see that 
$$
\C\{v_{i,1},v_{i, \lambda_{i}}^{*} \, | \, 1 \leq i \leq m\} = \C\{v_{i,1},v_{i, \lambda_{i}}^{*} \, | \, 1 \leq i \leq m-1\} \oplus \C\{v_{m,1}\}+W \subset x^{\mu_m+\nu_{m-1}-1}(W^{\perp}) + \C[x]v +W, 
$$
and so $X \subseteq x^{\mu_m+\nu_{m-1}-1}(W^{\perp}) + \C[x]v + W$. Therefore $k_2 \geq \mu_m+\nu_{m-1}$. But it is also clear that $\mu_m+\nu_{m-1}-1$ is the highest power of $x$ such that $v_{i,\mu_i-\mu_m+1}$ is contained in $x^{\mu_m+\nu_{m-1}-1}(W^{\perp})$ for $i \in \Delta_{m-1}$, and without these vectors, we cannot obtain the vector $v_{m,1} \in x^{\mu_m+\nu_{m-1}-1}(W^{\perp}) + \C[x]v$. Hence $k_2 = \mu_m+\nu_{m-1}$ and we are done. 
 \end{proof}
 
\begin{proof}[Proof of Theorem \ref{theorem:etypeY}: Case 2(a), $\nu_{m-1}>\nu_m \neq 0$] We now explicitly determine how $\sigma'$ is determined from $\sigma$. Suppose that $m >1$. By Propositions \ref{proposition:genericl2} and \ref{proposition:k2delta1} we have that $l_2=\lambda_m-1=\mu_m+\nu_m-1$ and $k_2=\mu_m+\nu_{m-1}$. Since $\nu_{m-1} > \nu_{m}$ we have 
$$
\sigma_{2(m-1)}=k_{2}=\mu_m+\nu_{m-1} > \mu_{m}+\nu_{m}-1=l_2=\sigma_{2m-1},
$$ 
so again $\sigma'$ is obtained from $\sigma$ by decreasing $\mu_m$ by $1$. Thus we have 
$$
\eType(v,x)\xrightarrow{\text{decrease $\nu_{m}$ by 1}}\eType(v+W,x_{|W^{\perp}/W})\xrightarrow{\text{decrease $\mu_{m}$ by 1}}\eType(v+Y,x_{|Y^{\perp}/Y}).
$$
In terms of the bumping algorithm, this corresponds to the case where the number moves across the wall from the right to the left, but remains on the same row. \vspace{5pt}

Now suppose that $m=1$. Then Propositions \ref{proposition:genericl2} and \ref{proposition:k2delta1} tells us that $\sigma'$ is obtained by reducing the last part of $\sigma$ of size $\lambda_1-1$ by $1$.
Since $\max \Gamma_1=\max \Delta_1$, this part must be $\sigma_1=\lambda_m-1=\mu_1+\nu_1-1$. Therefore $\sigma'$ is obtained from $\sigma$ by decreasing $\mu_1$ by $1$ and so in this case we have: 
$$
\eType(v,x)\xrightarrow{\text{decrease $\nu_{1}$ by 1}}\eType(v+W,x_{|W^{\perp}/W})\xrightarrow{\text{decrease $\mu_{1}$ by 1}}\eType(v+Y,x_{|Y^{\perp}/Y}).
$$
In terms of the bumping algorithm, this corresponds to the case where the number moves across the wall from the row on the right to the first row left. \end{proof}

\subsubsection{ $\mu'= \mu, \nu' \preceq \nu$ and either $\max\Gamma_m > \max \Delta_m$ or $\mu_m=0$.} 
We may assume without loss of generality that $m=\max \Delta_m$. If $\max \Gamma_m > \max \Delta_m$, then there is no restriction on the $\beta_i$; however if $\max \Gamma_m =  \max \Delta_m$, then $\sum_{i \in \Delta_m} \beta_i \neq 0$. Also, recall that in the case of $\mu_m=0$ (or equivalently that $\ell(\mu) < m$), then by the overarching assumptions that we are in the setting of Proposition \ref{proposition:perpendicular1}, we have $\nu_m \geq 2$. 

\begin{proposition} \label{proposition:k2d0}
Suppose $\eType(v+W, x_{|W^{\perp}/W})$ is obtained from $(\mu,\nu)$ be decreasing $\nu_m$ by $1$. Then the maximal $k_2$ such that $X \subseteq x^{k_2-1}(W^{\perp}) + \C[x]v + W$ is $\lambda_m-1$. 
\end{proposition}

\begin{proof}
In this case there is no restriction on the $\beta_i$ and Proposition \ref{proposition:maxk2prim} applies to tell us that $k_2 \geq \lambda_m-1$. First suppose that $\max \Gamma_m >m$. Then adding the space $\C[x]v$ will only result in adding the $\sum_{i=\min \Delta_m}^{\max \Gamma_m} v_{i,1}$ and thus we cannot obtain the individual $v_{i,1}$ for $i \in \Delta_m$. Hence  $X \subseteq x^{k-1}(W^{\perp}) + \C[x]v + W$ if and only if  $X \subseteq x^{k-1}(W^{\perp})$ and so $k_2=\lambda_m-1$. \vspace{5pt}

Similarly if we suppose that $\mu_m=0$, which is equivalent to $\ell(\mu) < m$, then it is clear that the vectors $v_{i,1}$ for $\ell(\mu)+1\leq i \leq m$ are contained in $x^{k-1}(W^{\perp}) + \C[x]v$ if and only if they are contained in $x^{k-1}(W^{\perp})$. This implies that the vectors  $v_{i,1}, v_{i,\lambda_i}^{*}$ for $\ell(\mu)+1\leq i \leq m$ are contained in $x^{k-1}(W^{\perp}) + \C[x]v+W$ if and only if they are contained in $x^{k-1}(W^{\perp})$ and so again it follows that $k_2=\lambda_m-1$. 
\end{proof}
\begin{proof}[Proof of Theorem \ref{theorem:etypeY}: Case 2(b), (c), (d)] \hspace{5mm}

\textbf{Case 2(b):} First suppose that $\nu_{m}-1 > \nu_{m+1}$. In the case where $\max \Gamma_m > \max \Delta_m$, we have $\mu_{m+1}=\mu_{m}$ and so $\max \Gamma_m=\max \Gamma_{m+1}$ and in the case where $\mu_m=0$ we again have $\mu_m=\mu_{m+1}$. By Propositions  \ref{proposition:genericl2} and \ref{proposition:k2d0} we have $k_2=l_2=\lambda_m-1$ and the last two parts $\sigma$ of this size are
$$
\sigma_{2m-1}=\mu_{m}+\nu_{m}-1 \quad \text{and} \quad \sigma_{2m}=\mu_{m+1}+\nu_{m}-1.
$$
Therefore it follows that $\sigma'$ is obtained from $\sigma$ by decreasing $\nu_{m}'=\nu_{m}-1$ by $1$. Thus we have 
$$
\eType(v,x)\xrightarrow{\text{decrease $\nu_{m}$ by 1}}\eType(v+W,x_{|W^{\perp}/W})\xrightarrow{\text{decrease $\nu_{m}'$ by 1}}\eType(v+Y,x_{|Y^{\perp}/Y}).
$$
In terms of the bumping algorithm, this corresponds to the number being removed from the right, moving one space to the right and remaining on the same row. \vspace{5pt}

Now suppose that $\nu_{m}-1=\nu_{m+1}$. Then $\mu_{m}+\nu_{m}-1=\mu_{m+1}+\nu_{m+1}=\lambda_{m+1}$ and so this further splits into two cases:
\begin{itemize}
\item \textbf{Case 2(c):} Suppose that $\max \Gamma_m > \max \Delta_{m+1}$ or $\mu_m=0$. Then the last two parts of $\sigma$ of size $\lambda_{m+1}$ are 
$$
\sigma_{2\max \Delta_{m+1}-1}=\mu_{\max \Delta_{m+1}}+\nu_{\max \Delta_{m+1}} \quad \text{and} \quad \sigma_{2\max \Delta_{m+1}}=\mu_{\max \Delta_{m+1}+1}+\nu_{\max \Delta_{m+1}},
$$
and so $\sigma'$ is obtained from $\sigma$ by decreasing $\nu_{\max \Delta_{m+1}}$ by $1$.  Note in the case where $\mu_m=0$ we have $\mu_{m+1}=0$ as well and so $\mu_{\max \Delta_{m+1}}=\mu_{\max \Delta_{m+1}+1}=0$.  This corresponds to the number moving from the right, moving down possibly several rows but remaining on the right. We have 
$$
\eType(v,x)\xrightarrow{\text{decrease $\nu_{m}$ by 1}}\eType(v+W,x_{|W^{\perp}/W})\xrightarrow{\text{decrease $\nu_{\max \Delta_{m+1}}$ by 1}} \eType(v+Y,x_{|Y^{\perp}/Y}).
$$
In terms of the bumping algorithm, this corresponds to a number moving from the right of the wall and remaining on the right of the wall, but moving down possibly many rows. 
\item \textbf{Case 2(d):} Now suppose that $\max \Gamma_m \leq \max \Delta_{m+1}$ or $\nu_{m+1}=0$. Then the last two parts of $\sigma$ of size $\lambda_m-1$ are: 
$$
\sigma_{2(\max \Gamma_m-1)}=\mu_{\max \Gamma_{m}}+\nu_{\max \Gamma_m-1} \quad \text{and} \quad \sigma_{2\max \Gamma_m}=\mu_{\max \Gamma_m}+\nu_{\max \Gamma_{m}}.
$$
Note that it is possible for $\nu_{\max \Gamma_{m}}=0$ but $\nu_{\max \Gamma_m-1}\neq 0$. This is certainly the case when $\nu_{m+1}$ and $\max \Gamma_m-1=\max \Delta_m=m$. In any case we obtain $\sigma'$ from $\sigma$ be decreasing $\mu_{\max \Gamma_m}$ by $1$. We have 
$$
\eType(v,x)\xrightarrow{\text{decrease $\nu_{m}$ by 1}}\eType(v+W,x_{|W^{\perp}/W})\xrightarrow{\text{decrease $\mu_{\max \Gamma_m}$ by 1}}\eType(v+Y,x_{|Y^{\perp}/Y}).
$$
 This corresponds to the number moving across the wall, moving down possibly several rows and resting on the left. 
\end{itemize} \end{proof}

\section{Proof of Main Theorem} \label{section:algorithmconstruction}
\subsection{Setup}
In this subsection we go over the notation needed in the proof of Theorem \ref{theorem:main}. 

\begin{definition}\label{def:Phi} Given $(v, x) \in \mO_{(\mu,\nu)}$, we say that a flag $F_{\bullet} \in \cC_{(v,x)}$ is `good' if the sequence of bi-partitions $\Phi(F_\bullet)$ below is a nested sequence (i.e. one box is removed at each stage): $$\Phi(F_\bullet)=( \eType(v,x), \eType(v+F_1, x|_{F_1^{\perp}/F_1}), \eType(v+F_2, x|_{F_2^{\perp}/F_2}), \cdots ).$$
Notice that by Theorem \ref{theorem:irreduciblecomponents} the set of good flags is open dense in $\cC_{(v,x)}$. \end{definition}
\begin{definition} Let $F_{\bullet}, G_{\bullet} \in \cC_{(v,x)}$ be two good flags: 
\begin{eqnarray*}
F_{\bullet}&=&(0 \subset F_1 \subset F_2 \subset \ldots \subset F_{2}^{\perp} \subset F_{1}^{\perp} \subset \C^{2n}), \\
G_{\bullet}&=&(0 \subset G_1 \subset G_2 \subset \ldots \subset G_{2}^{\perp} \subset G_{1}^{\perp} \subset \C^{2n}), 
\end{eqnarray*}
We define two flags in the smaller exotic fibre $\cC_{(\overline{v},\overline{x})}$ where $\overline{x}=x|_{G_1^\perp /G_1}$ and $\overline{v}=v+G_1\in G_1^\perp/G_1$. 
\begin{eqnarray}
\widetilde{F_{\bullet}}&=& \left( \ldots \subset \frac{G_1+F_{i} \cap G_{1}^{\perp}}{G_1} \subset \ldots \right)_{i=0}^{2n}, \\
\widetilde{G_{\bullet}}&=&(0 = G_1/G_1 \subset G_2/G_1 \subset \ldots \subset G_{2}^{\perp}/G_1 \subset G_{1}^{\perp}/G_1 \cong \C^{2(n-1)}). 
\end{eqnarray}
Notice that the flag $\widetilde{F_{\bullet}}$ will have redundancies in two places, i.e. there are two numbers $k$ such that $\widetilde{F}_k=\widetilde{F}_{k+1}$.

If $F_n \subset G_1^{\perp}$ (or, equivalently, if $F_n\supset G_1$) we call this a Type 1 redundancy. Let $r$ be minimal such that $F_{n+r}\not\subset G_{1}^{\perp}$. It follows then that $F_{n+r-1}=F_{n+r} \cap G_{1}^{\perp}$ and also that $G_1+F_{n-r}=F_{n-r+1}$.  In this case, the flag $\widetilde{F_{\bullet}}$ looks like: 
\begin{footnotesize}$$\widetilde{F_{\bullet}}=\left( 0 \subset \frac{G_1+F_1}{G_1} \subset \ldots \subset \frac{G_1+F_{n-r-1}}{G_1} \subset \frac{F_{n-r+1}}{G_1} \subset \ldots \subset \frac{F_{n+r-1}}{G_1} \subset \frac{F_{n+r+1}\cap G_{1}^{\perp}}{G_1} \subset \ldots \subset  \frac{G_{1}^{\perp}}{G_1}\right).$$ 
\end{footnotesize}
In this flag the indices $n+r$ and $n-r$ are missing. 

If $F_n \not\subset G_1^{\perp}$ ($F_n\not\supset G_1$),
we call this a Type 2 redundancy. Let $r$ be minimal such that $F_{n+r}\supset G_1$, we have $F_{n+r-1}+G_1=F_{n+r}$ and $F_{n-r-1}=F_{n-r} \cap G_{1}^{\perp}$, and: 
\begin{scriptsize}
\begin{displaymath}
\widetilde{F_{\bullet}}=\left( 0 \subset  \frac{G_1+F_{1}}{G_1} \subset \ldots \subset  \frac{G_1+F_{n-r-1}}{G_1} \subset  \frac{G_1+F_{n-r+1} \cap G_{1}^{\perp}}{G_1} \subset \ldots \subset  \frac{G_1+F_{n+r-1} \cap G_{1}^{\perp}}{G_1} \subset  \frac{F_{n+r+1} \cap G_{1}^{\perp}}{G_1} \subset \ldots \subset \frac{G_{1}^{\perp}}{G_1}\right).
\end{displaymath} 
\end{scriptsize}
Again the indices $n+r$ and $n-r$ are missing from the labelling. 

In both cases, since $F_\bullet$ was a good flag, so is $\widetilde{F}_\bullet$. Keeping track of the redundancies in the labelling, we can then say that $\Phi(\widetilde{F}_\bullet)$ is a bitableau that satisfies the increasing (standard) condition, containing the numbers $1,\ldots,n$ excluding $r$. 
\end{definition}

Let $F_\bullet\in\cC_{(v,x)}$ be a good flag. By Theorem \ref{theorem:irreduciblecomponents},  for any $s$ with $1 \leq s \leq n-1$ we have a bijective map 
\begin{equation} \label{equation:restrictionmap}
\Phi_{s}: \Irr \cC_{(v+G_{n-s}, x_{|G_{n+s}/G_{n-s}})} \longrightarrow \cT(\epsilon^{(s)},\eta^{(s)})
\end{equation}
where $(\epsilon^{(s)},\eta^{(s)})=\eType(v+G_{n-s}, x_{|G_{n+s}/G_{n-s}})$.
\begin{definition}We define the $s$-\emph{truncated flag} $$F_{\bullet}^{s}:=\left(0=\frac{F_{n-s}}{F_{n-s}} \subset \frac{F_{n-s+1}}{F_{n-s}} \subset \ldots \subset \frac{F_{n+s-1}}{F_{n-s}}  \subset \frac{F_{n+s}}{F_{n-s}} \cong \C^{2s} \right)$$ We have that $F_\bullet^s\in\cC_{(v+G_{n-s}, x_{|G_{n+s}/G_{n-s}})}$ is a good flag, and  $\Phi_{s}(F_{\bullet}^{s})=T_s$, the bitableau obtained from $T$ by only considering the number $1$ up to $s$ (see Definition \ref{def:truncT}). \end{definition}

\begin{definition}Let $F_\bullet$, $G_\bullet\in\cC_{(v,x)}$ be two good flags, and set $\Phi(F_{\bullet})=T$ and $\Phi(G_{\bullet})=R$. Under the map $\Phi_{n-1}$, we have $\Phi_{n-1}(\widetilde{G_{\bullet}})=R_{n-1}$. Define $\widetilde{T}$ to be the standard bitableau obtained by relabelling the entries of $\Phi_{n-1}(\widetilde{F_{\bullet}})$ with numbers from $1$ up to $n$ with $r$ missing, as discussed above. 
Similarly, we define the flag $\widetilde{F_{\bullet}^{s}}$ and the tableau $\widetilde{T}_s$ and we have $\Phi_{n-s+1}(\widetilde{F_{\bullet}^{s}})=\widetilde{T}_s$. \end{definition}

\begin{lemma} \label{lemma:r=n}
Suppose that $r=n$ in the index determining the redundancies of the flag $\widetilde{F}_{\bullet}$. Then $\widetilde{F}_{\bullet} \cong F_{\bullet}^{n-1}$. 
\end{lemma}
\begin{proof}
In the case of the Type 1 redundancy, we have $F_{i} \subseteq G_{1}^{\perp}$ for all $1 \leq i \leq 2n-1$, which implies that $F_{2n-1}=G_1^\perp$ and $F_1=G_1$. Therefore each term in $\widetilde{F_{\bullet}}$ has the form $F_{i}/F_{1}$ except for the last term which is $(\C^{2n}\cap G_{1}^{\perp})/G_1=G_{1}^{\perp}/G_1$ and so   $\widetilde{F}_{\bullet} = F_{\bullet}^{n-1}$. \vspace{5pt}

In the case of the Type 2 redundancy, we have $G_1 \not \subseteq F_{i}$ for $1 \leq i \leq 2n-1$. Therefore $F_1 \cap G_{1}^{\perp}=0$ and so the first non-zero term in $\widetilde{F_{\bullet}}$ is $\frac{G_1+F_2 \cap G_{1}^{\perp}}{G_1}\cong F_2 \cap G_{1}^{\perp} \cong F_{2}/F_{1}$ and every subsequent term is isomorphic to $$\frac{G_1+F_i \cap G_{1}^{\perp}}{G_1} \cong F_i \cap G_{1}^{\perp} \cong \frac{F_i \cap G_{1}^{\perp}+F_1}{F_1}\cong F_{i}/F_{1}$$ except for the last term which is $(\C^{2n} \cap G_{1}^{\perp})/G_1=G_{1}^{\perp}/G_1$. So in this case we also have $\widetilde{F}_{\bullet} \cong F_{\bullet}^{n-1}$.
\end{proof}

\subsection{Inductive step}

We now use the techniques developed by Steinberg in \cite{Ste} to determine the exotic Robinson-Schensted algorithm; this section is devoted to the proof of the following key fact (analogous to Lemma 1.2 from \cite{Ste}), which is the inductive step in the proof. 

\begin{definition} We say that two points $F_{\bullet}, G_{\bullet} \in \mathcal{C}_{(v,x)}$ are ``compatible'' if $F_{\bullet}, G_{\bullet}, \widetilde{F}_{\bullet}$ and $\widetilde{G}_{\bullet}$ are all good flags.
 \end{definition}

\begin{proposition} \label{proposition:mainlemma}
Let $F_{\bullet}$ and $G_{\bullet}$ be two compatible points in the exotic Springer fibre $\cC_{(v,x)}$. Let $T$ and $R$ be the bitableaux corresponding to  $F_{\bullet}$ and $G_{\bullet}$ and let $\widetilde{T}$ and $R_{n-1}$ be the bitableaux corresponding to $\widetilde{F}_{\bullet}$ and $\widetilde{G}_{\bullet}$. Then the pair $(\widetilde{T}, R_{n-1})$ is obtained from $(T,R)$ by the first step of the reverse bumping algorithm described in Section \ref{section:bumpingalgorithm}. 
\end{proposition}

As in the Steinberg proof, at the first stage of the algorithm, some number $s \in T$, which is in the same position as $n \in R$ moves and causes a cascade which eventually bumps a number $r$ off. We need to determine where this number $s$ lies in the new bitableau $\widetilde{T}$. We break the steps of the Steinberg proof into smaller lemmas. 

\begin{lemma} \label{lemma:rlessthans}
The numbers $1, \ldots, r-1$ occupy the same positions in $\widetilde{T}$ as they do in $T$.  
\end{lemma} 

\begin{proof}
It suffices to show that for $k < r$, the flags $F_{\bullet}^{k}$ and $\widetilde{F_{\bullet}^{k}}$ are naturally isomorphic.
The flag $F_{\bullet}^{k}$ consists of spaces of the form $F_{n \pm k'}/F_{n-k}$, where $k' \leq k$. Suppose first that we have a Type 1 redundancy. Then the spaces in the flag $\widetilde{F_{\bullet}^{k}}$ have the form $$\frac{F_{n \pm k'}}{G_1} \bigg/ \frac{F_{n - k}}{G_1} \cong F_{n \pm k'}/F_{n-k},$$ where $k' \leq k$ and so $F_{\bullet}^{k}$ and $\widetilde{F_{\bullet}^{k}}$ map to the same bitableau. \vspace{10pt}

Now suppose that we are in a Type 2 redundancy. Then a subspace in $\widetilde{F_{\bullet}^{k}}$ has the form $$\frac{G_1 + F_{n \pm k'} \cap G_{1}^{\perp}}{G_1} \bigg/ \frac{G_1 + F_{n-k} \cap G_{1}^{\perp}}{G_1} \cong \frac{G_1 + F_{n \pm k'} \cap G_{1}^{\perp}}{G_1 + F_{n - k} \cap G_{1}^{\perp}},$$ where $k' \leq k$. Now since $G_{1}$ is not contained in $F_{n-k}$ we have $$\frac{G_1 + F_{n \pm k'} \cap G_{1}^{\perp}}{G_1 + F_{n - k} \cap G_{1}^{\perp}} \cong \frac{F_{n \pm k'} \cap G_{1}^{\perp}}{F_{n - k} \cap G_{1}^{\perp}}.$$ Also, there is a natural map $F_{n \pm k'} \cap G_{1}^{\perp} \longrightarrow F_{n \pm k'}/F_{n-k}$, which is the composition of the natural inclusion $F_{n \pm k'} \cap G_{1}^{\perp} \hookrightarrow F_{n \pm k'}$ and the natural projection $F_{n \pm k'} \rightarrow F_{n \pm k'}/F_{n-k}$, whose kernel is $F_{n-k} \cap G_{1}^{\perp}$. Therefore $$\frac{F_{n \pm k'} \cap G_{1}^{\perp}}{F_{n - k} \cap G_{1}^{\perp}} \cong F_{n \pm k'}/F_{n-k},$$ and so again, $F_{\bullet}^{k}$ and $\widetilde{F_{\bullet}^{k}}$ map to the same bitableau. 
\end{proof}

From now on we may assume that $s \geq r$. In fact, if $s=r$ then a special case of Lemma \ref{lemma:r=n} tells us that $\widetilde{T}_{s}=T_{s-1}$, which we will emphasize below. 

\begin{lemma} \label{lemma:rs}
With $s \geq r$ as above we have:
\begin{enumerate}[(1)]
\item $s=r$ if and only if $s$ is in the $(m,1)$-th position of $T_{s}$ with either 
\begin{enumerate}[(a)]
\item $m=1$; or 
\item $m>1$ and either $\mu_{m}^{(s)}=1$ and $\nu_{m-1}^{(s)}=0$, or $\nu_{m}^{(s)}=1$ and $\mu_{m}^{(s)}=0$.
\end{enumerate}
\item When $s>r$, $s$ moves from $T_s$ into $\widetilde{T_s}$, where it displaces a smaller number. 
\end{enumerate}
\end{lemma}

\begin{proof}
Consider the truncated flags $F_{\bullet}^{s}$ and $\widetilde{F_{\bullet}^{s}}$ defined above. These flags map to bitableaux $T_s$ containing the numbers $1$ up to $s$ in the first case and $\widetilde{T_s}$ containing the numbers  $1$ up to $s$ with $r$ removed in the second case. \vspace{10pt}

We assume without loss of generality that we are in a Type 1 redundancy, as the Type 2 case is completely analogous.
We have that $F_{n+s} \cap G_{1}^{\perp}$ is contained in $F_{n+s}$ as a hyperplane. Moreover, both spaces contain $F_{n-s}$ as a subspace. Therefore $$\frac{F_{n+s} \cap G_{1}^{\perp}}{F_{n-s}} \subset F_{n+s}/F_{n-s}$$ as a hyperplane and, inside the symplectic space $F_{n+s}/F_{n-s}$, we have that $$\left(\frac{F_{n+s} \cap G_{1}^{\perp}}{F_{n-s}}\right)^{\perp}= \frac{F_{n-s}+G_1}{F_{n-s}},$$
which is a 1-dimension subspace of $F_{n+s}/F_{n-s}$ contained in the $\ker(x_{|_{F_{n+s}/F_{n-s}}})$. \vspace{10pt}

Now $$\frac{F_{n+s} \cap G_{1}^{\perp}}{F_{n-s}} \bigg/ \frac{F_{n-s}+G_1}{F_{n-s}} \cong \frac{F_{n+s} \cap G_{1}^{\perp}}{F_{n-s}+G_1},$$ which is the last term in the flag $\widetilde{F_{\bullet}^{s}}$. This shows that $(\frac{F_{n+s} \cap G_{1}^{\perp}}{F_{n-s}})^{\perp}=\frac{F_{n-s} + G_{1}}{F_{n-s}}$ is a one-dimensional space in $\ker(x_{|_{F_{n+s}/F_{n-s}}})$. 
But $(F_{n+s-1}/F_{n-s})^{\perp}=F_{n-s+1}/F_{n-s}$ is another such space generic with respect to this property. \vspace{5pt}
Now if $s$ is in the $(1,1)$-th position of $T_{s}$, that is, $m=1$,  then by Corollary \ref{corollary:uniquegeneric}, it follows immediately that  $\frac{F_{n+s} \cap G_{1}^{\perp}}{F_{n-s}}= F_{n+s-1}/F_{n-s}$, and so $F_{n+s} \cap G_{1}^{\perp}=F_{n+s-1}$. Therefore $F_{n+s-1} \subset G_{1}^{\perp}$ but  $F_{n+s} \not \subset G_{1}^{\perp}$ and so $r=s$ by choice of $r$ and we are in a Type 1 redundancy. \vspace{5pt}

If $s$ is in the $(m,1)$-th position of $T_{s}$ with $m>1$ with either $\mu_{m}^{(s)}=1$ and $m> \ell(\nu_{m}^{(s)})+1$, or $\nu_{m}^{(s)}=1$ and $m>\ell(\mu_{m}^{(s)})$, then by Corollary \ref{corollary:notperp}
since in both cases we have $\lambda_m=1$, we have $F_{n-s+1}/F_{n-s} \not \subseteq \frac{F_{n+s} \cap G_{1}^{\perp}}{F_{n-s}}$. Therefore $F_{n-s+1}\not \subseteq F_{n+s} \cap G_{1}^{\perp}$, which implies that $F_{n-s+1}\not \subseteq  G_{1}^{\perp}$ and so we are in a Type 2 redundancy. Rewriting this condition as $F_{n+s-1}\not \supseteq  G_{1}$ and remembering the choice of $r$ in this case we conclude that $n+s-1 \leq n+r-1$, that is $s \leq r$. But our original assumption is that $s \geq r$ and so we have $s=r$. \vspace{5pt}

Conversely if $s=r$, then the above arguments show that either $F_{n+s} \cap G_{1}^{\perp}=F_{n+s-1}$ in the Type 1 redundancy or $F_{n-s+1}\not \subseteq F_{n+s} \cap G_{1}^{\perp}$ in the Type 2 redundancy. It then follows by Corollaries \ref{corollary:notperp} and \ref{corollary:uniquegeneric} below that $s$ is in the $(m,1)$-th position with either $m=1$ in the Type 1 redundancy case, or $m>1$ and either $\mu_{m}^{(s)}=1$ and $m> \ell(\nu_{m}^{(s)})+1$, or $\nu_{m}^{(s)}=1$ and $m>\ell(\mu_{m}^{(s)})$ in the Type 2 redundancy. Therefore parts (a) and (b) of (1) are proved. In these cases, the number $r$ simply disappears from $T_{r}$ and so $\widetilde{T}_{r}=T_{r-1}$. \vspace{5pt}

Now let us assume that $s > r$. Then $F_{n-s+1}/F_{n-s} \subseteq \frac{F_{n+s} \cap G_{1}^{\perp}}{F_{n-s}}$ and so the number $s$ moves from $T_s$ into $\widetilde{T_s}$. By a similar argument, every number less than $s$, excluding $r$ moves from  $T_s$ into $\widetilde{T_s}$. Now by Theorem \ref{theorem:etypeY}, (the exotic analogue of the \cite[Lemma 3.2]{Ste}), with $W=\frac{F_{n-s} + G_{1}}{F_{n-s}}$, $X=F_{n-s+1}/F_{n-s}$ and so $Y:=X+W=\frac{F_{n-s+1} + G_{1}}{F_{n-s}} \subset \frac{F_{n+s-1} \cap G_{1}^{\perp}}{F_{n-s}}=X^{\perp} \cap W^{\perp} =Y^{\perp}$, we know how $$\eType\left(v+(F_{n-s+1}+G_1), x_{|_{\frac{F_{n+s-1} \cap G_{1}^{\perp}}{F_{n-s+1}+G_1}}}\right)$$ is obtained from $\eType(v+F_{n-s},x_{|F_{n+s}/F_{n-s}})$, which tells us what position $s$ occupies in $\widetilde{T_s}$. Therefore, $s$ has moved from $T_{s}$ into  $\widetilde{T_s}$, where it has displaced a smaller number $s' < s$. Therefore (2) holds and we are done.  
\end{proof}

\begin{proof}[Proof of Proposition \ref{proposition:mainlemma}]
We let $s$ be in the $(m,1)$-th or $(m,\lambda_m)$-th position of $T$. By Lemmas \ref{lemma:rlessthans} and \ref{lemma:rs}, we know that in the transition from $T$ to $\widetilde{T}$, either $s$ disappears from the bitableau or moves and displaces a smaller number, which in turn displaces a smaller number and so on, until eventually $r$ disappears from the tableau (from the first row in the Type 1 redundancy, or from a column immediately to the left or right of the dividing wall in the Type 2 redundancy). This chain of displacements is governed by the rules of Theorem \ref{theorem:etypeY}. It also  follows by Theorem \ref{theorem:etypeY}, that no other number in $T$ causes a chain of displacements as that would imply that different numbers in $T_s$ would move to the same position of $\widetilde{T}_{s}$, which is impossible. Since, as explained in Remark \ref{remark:43}, Theorem \ref{theorem:etypeY} is the geometric incarnation of the reverse bumping algorithm described in Section \ref{section:bumpingalgorithm}, it follows that the first step of the Exotic Robinson-Schensted bijection has been achieved. 
\end{proof}

\subsection{Completing the proof.}

While it may seem that the main theorem follows from Proposition \ref{proposition:mainlemma} by induction, there is a subtlety: $\widetilde{F}_{\bullet}, \widetilde{G}_{\bullet}$ may not be generic, even though $F_{\bullet}, G_{\bullet}$ are. 
In this section we remedy this using the argument on pg 528 of Steinberg, \cite{Ste}. 

\begin{remark}Recall the definition of relative position of two symplectic flags $F_\bullet, G_\bullet\in\cF(V)$ from Definition \ref{relative}, then 
the Bruhat decomposition states that the orbits of $Sp_{2n}(\mathbb{C})$ on $\cF(V) \times \cF(V)$ are determined by the relative positions of the two flags. We have the Bruhat ordering on $W(C_n)$, which satisfies the following: for each $w \in W(C_n)$, the set of all flags $(F_{\bullet}, G_{\bullet})$ that satisfy $w(F_{\bullet}, G_{\bullet}) \leq w$ is a closed set. 
\end{remark}

\begin{definition} \label{definition:neww}
Given $w \in W(C_n)$, define $\tilde{w} \in W(C_{n-1})$ as follows. Suppose that $w(n) \in \{r, \overline{r}\}$, then for $1 \leq i \leq n-1$ define $\tilde{w}(i) = w(i)$ if $w(i) < r$, and $\tilde{w}(i) = w(i+1)-1$ if $w(i+1) > r$.  
\end{definition}

\begin{lemma} \label{lemma:relativeposition}
If $w = w(F_{\bullet}, G_{\bullet})$, then $w(\widetilde{F_{\bullet}}, \widetilde{G_{\bullet}}) = \tilde{w}$. 
\end{lemma} 
\begin{proof} Recall from Definition \ref{relative} that if $w = w(F_{\bullet}, G_{\bullet})$, then we can choose an orthonormal basis $\{ v_1, \cdots, v_n, \overline{v}_n, \cdots, \overline{v}_{1} \}$ (i.e. so that $\langle v_i, \overline{v}_j \rangle = \delta_{i,j}=-\langle  \overline{v}_j, v_i \rangle, \langle v_i, v_j \rangle = \langle \overline{v}_i, \overline{v}_j \rangle = 0 $), with 

\begin{align*}
F_{i} &= \C \{ v_{n}, \ldots, v_{n-i+1}\},	&F_{2n-i}=F_{i}^{\perp},    \\ 
G_{j} &= \C\{v_{w(n)}, \ldots, v_{w(n-j+1)}\},  	&G_{2n-j}=G_{j}^{\perp},
\end{align*}
for all $1 \leq i,j \leq n$. Using this basis we can explicitly describe the flags $\widetilde{G}_{\bullet}$ and $\widetilde{F}_{\bullet}$, while being careful regarding the type of redundancy that determines $\widetilde{F}_{\bullet}$. Firstly, we have
\begin{equation} \label{equation:newG}
\widetilde{G}_{\bullet} = (0 \subset \mathbb{C}\{ v_{w(2)} \} \subset \mathbb{C}\{ v_{w(2)}, v_{w(3)} \} \subset \cdots \subset \mathbb{C} \{ v_{w(2)}, v_{w(3)}, \cdots, v_{w(2n-2)}, v_{w(2n-1)} \} ).
\end{equation}
Now for $\widetilde{F}_{\bullet}$ we had $\widetilde{F}_{i}=\frac{G_1 + F_{i} \cap G_{1}^{\perp}}{G_1}$ for each $i$. In the Type 1 redundancy we had $F_n \subseteq G_{1}^{\perp}$ and chose $r$ minimal such that $F_{n+r} \not \subseteq G_{1}^{\perp}$. Therefore 
$$
F_{n+r}=\C\{ v_{n},\ldots, v_{1}, v_{\bar{1}},\ldots,v_{\bar{r}} \} \not\subseteq G_{1}^{\perp} \quad \text{but} \quad F_{n+r-1}=\C\{v_{n},\ldots, v_{1}, v_{\bar{1}},\ldots,v_{\overline{r-1}} \} \subseteq G_{1}^{\perp},
$$
so this implies that $w(\bar{n})=\bar{r}$, or equivalently $w(n)=r$. For the Type 2 redundancy, we chose $r$ minimal such that $F_{n+r} \supseteq G_1$, so in this case: 
$$
F_{n+r}=\C\{ v_{n},\ldots, v_{1}, v_{\bar{1}},\ldots,v_{\bar{r}} \} \supseteq G_{1}=\C\{v_{w(n)}\} \quad \text{but} \quad F_{n+r-1}=\C\{v_{n},\ldots, v_{1}, v_{\bar{1}},\ldots,v_{\overline{r-1}} \} \not\supseteq G_{1},
$$
which implies $w(n)=\bar{r}$. 
Explicitly, in both cases, we have: 
$$
\widetilde{F}_{\bullet,i} = 
\begin{cases}
\C\{ v_{n},\ldots,  v_{n-i+1} \} 						&\text{if} \, \, i  < n-r+1,  \\
\C\{ v_{n},\ldots,  \hat{v}_{r}, \ldots,v_{n-i+1} \} 			&\text{if} \, \,  i > n-r+1,
\end{cases}
$$
where $\hat{v}_{r}$ denotes omission of $v_{r}$, and $\widetilde{F}_{\bullet,2n-i} = \widetilde{F}_{\bullet,i}^{\perp}$. \vspace{5pt}
After apply the transformation $ \theta: \{1,\ldots, \hat{r}, \ldots, n\} 	\longrightarrow  \{1,\ldots, n-1\}$ defined by 
$$
\theta(i)=
\begin{cases}
i 	&\text{if} \quad i < r  \\
i-1	&\text{if} \quad i > r,
\end{cases}
$$ and replacing $w$ by $\tilde{w}$ as defined in Definition \ref{definition:neww} in the expression \eqref{equation:newG} for $\widetilde{G}_{\bullet}$, the lemma follows. 
\end{proof}

\begin{remark}
For notation convenience, we will think of $\tilde{w}$ as a signed bijection from the set $\{1,\ldots, n-1\} \rightarrow \{1, \ldots, \hat{r}, \ldots n\}$, but still regard it as an element of $W(C_{n-1})$. 
\end{remark}

\begin{lemma} \label{lemma:RSinduction}
Suppose that  $\RS(T, R) = w$, then $\RS(\widetilde{T}, R_{n-1}) = \widetilde{w}$.  

\end{lemma} 
\begin{proof} This follows from the reverse bumping algorithm. The moves performed when removing $n-1$ from the pair $(T_{n-1}, R_{n-1})$ exactly mirror the moves performed when removing $n-1$ from $(T^{(n-1)}, T^{(n-1)})$ (though the labelling is slightly different). The conclusion follows by iterating this argument. 
\end{proof}


\begin{proof}[Proof of Theorem \ref{theorem:main}] 
To finish off the proof it only remains for us to examine the case where $\widetilde{F}_{\bullet}$ is not generic relative to the bitableau $\widetilde{T}$ even if $F_{\bullet}$ and $G_{\bullet}$ were chosen generically relative to $T$ and $R$ respectively.  \vspace{5pt}

Let $\RS(T,R)=w$. By induction, the relative position of a generic element lying in $\Phi_{n-1}^{-1}(\widetilde{T})$, and a generic element lying in $\Phi_{n-1}^{-1}(S_{n-1})$ is equal to $\RS(\widetilde{T}, R_{n-1}):=\widetilde{w}$, thought of as a signed bijection $\{1,\ldots, n-1\}$ to $\{1,\ldots, \hat{r}, \ldots, n\}$ by simply deleting the last letter of $w$. We know by Proposition \ref{proposition:mainlemma} that the pair of bitableaux $(\widetilde{T},R_{n-1})$ is obtained from $(T,R)$ by the Exotic Robinson-Schensted algorithm.
\vspace{5pt}

Now since the set of all such flags $(V_{\bullet}), (W_{\bullet})$ satisfying $w(\widetilde{V}_{\bullet}, \widetilde{W}_{\bullet}) \leq \widetilde{w}$ is a closed set, we may assume that 
$w(\widetilde{F}_{\bullet},\widetilde{G}_{\bullet}) \leq \widetilde{w}$, since if $\widetilde{F}_{\bullet}$ is not generic relative to $\widetilde{T}$, we know that it lies in the closure of all flags $\widetilde{F'}_{\bullet}$ in the Exotic Springer Fibre that are (see \cite[Theorem 2.12]{NRS16}). Now the inequality is still satisfied when we put $r$ at the end of the word $\widetilde{w}$. Therefore, since inserting $r$ preserves the Bruhat ordering on $W$, it follows that $w(F_{\bullet}, G_{\bullet}) \leq \RS(S, T)$. \vspace{5pt}

As $(T, R)$ varies over the bitableau, the relative positions $w(F_{\bullet}, G_{\bullet})$ pass over each permutation exactly once, this is a fact about the exotic Steinberg variety which follows from \cite[Lemma 1.5]{Kat}, or more generally from \cite[Theorems 3.5 and 3.6]{SteinRSK}. Further, their images under the exotic Robinson-Schensted bijection $w(T, R)$ also sweep out each permutation exactly once, since the correspondence is a bijection. It now follows that the inequality $w(F_{\bullet}, G_{\bullet}) \leq \RS(T, S)$ is in fact an equality (by backwards induction on the length of the word $\RS(T,R)$, the base case being the long word $w_0$). Hence the main theorem is proved. \end{proof}

\newpage

\begin{appendix}

\section{The exotic Robinson-Schensted bijection for $n=3$}
Here we give the complete exotic Robinson-Schensted bijection for $n=3$. The $n=2$ case was given in \cite{NRS16}. We will give this correspondence in such a way that the exotic cells in the Weyl group are clear. 

\scriptsize{
\begin{longtable}{|c c | c c |}
\hline
\textbf{Element of $W(C_3)$} & \textbf{Bitableaux} 	&\textbf{Element of $W(C_3)$} 					& \textbf{Bitableaux}  \\
\hline  
&&&\\
$123$ 					&  $\left(\left(\young(321); - \right),\left(\young(321); -\right)\right)$ 		&  $\bar{3}\bar{2}\bar{1}$ 	& $\left(\left(-;\young(123)\right), \left(-;\young(123) \right)\right)$ \\
&&&\\
\hline
&&& \\
$1\bar{2}\bar{3}$ 			& $\left(\left(\young(1,2,3); - \right), \left(\young(1,2,3); - \right)\right)$  	&  $\bar{1}\bar{2}\bar{3}$ 	& $\left(\left(-;\young(1,2,3)\right), \left(-;\young(1,2,3) \right)\right)$ \\
&&&\\
\hline 
&&& \\

$1\bar{2}3$				& $\left(\left(\young(31,:2); - \right),\left(\young(31,:2); -\right)\right)$		&$\bar{1}\bar{3}\bar{2}$	& $\left(\left(-;\young(13,2) \right),\left(-;\young(13,2)\right)\right)$ \\

$12\bar{3}$				& $\left(\left(\young(21,:3); - \right),\left(\young(21,:3); -\right)\right)$ 		& $\bar{2}\bar{1}\bar{3}$		& $\left(\left(-;\young(12,3) \right),\left(-;\young(12,3)\right)\right)$ \\
 
$1\bar{3}2$ 				&  $\left(\left(\young(21,:3); - \right), \left(\young(31,:2); - \right)\right)$ 	&$\bar{2}\bar{3}\bar{1}$		& $\left(\left(-;\young(12,3) \right),\left(-;\young(13,2)\right)\right)$ \\

$13\bar{2}$				& $\left(\left(\young(31,:2); - \right), \left(\young(21,:3); - \right)\right)$ 	&$\bar{3}\bar{1}\bar{2}$		& $\left(\left(-;\young(13,2) \right),\left(-;\young(12,3)\right)\right)$ \\
&&&\\
\hline
&&& \\
$132$				& $\left(\left(\young(21);\young(3)\right), \left(\young(21);\young(3)\right)\right)$ 	&$3\bar{2}1$			& $\left(\left(\young(1);\young(23)\right), \left(\young(1);\young(23)\right)\right)$ 	  \\

$213$				& $\left(\left(\young(31);\young(2)\right), \left(\young(31);\young(2)\right)\right)$ 	&$\bar{1}32$			& $\left(\left(\young(2);\young(13)\right), \left(\young(2);\young(13)\right)\right)$ 	  \\

$\bar{1}23$			& $\left(\left(\young(32);\young(1)\right), \left(\young(32);\young(1)\right)\right)$ 	&$\bar{2}\bar{1}3$		& $\left(\left(\young(3);\young(12)\right), \left(\young(3);\young(12)\right)\right)$ 	  \\

$312$				& $\left(\left(\young(21);\young(3)\right), \left(\young(31);\young(2)\right)\right)$ 	&$\bar{2}31$			& $\left(\left(\young(1);\young(23)\right), \left(\young(2);\young(13)\right)\right)$ 	  \\

$231$				& $\left(\left(\young(31);\young(2)\right), \left(\young(21);\young(3)\right)\right)$ 	&$3\bar{1}2$			& $\left(\left(\young(2);\young(13)\right), \left(\young(1);\young(23)\right)\right)$ 	  \\

$\bar{3}12$			& $\left(\left(\young(21);\young(3)\right), \left(\young(32);\young(1)\right)\right)$ 	&$\bar{3}\bar{2}1$		& $\left(\left(\young(1);\young(23)\right), \left(\young(3);\young(12)\right)\right)$ \\

$23\bar{1}	$			& $\left(\left(\young(32);\young(1)\right), \left(\young(21);\young(3)\right)\right)$ 	&$3\bar{2}\bar{1}$		& $\left(\left(\young(3);\young(12)\right), \left(\young(1);\young(23)\right)\right)$ \\

$\bar{2}13$			& $\left(\left(\young(31);\young(2)\right), \left(\young(32);\young(1)\right)\right)$ 	&$\bar{3}\bar{1}2$		& $\left(\left(\young(2);\young(13)\right), \left(\young(3);\young(12)\right)\right)$ \\

$2\bar{1}3$			& $\left(\left(\young(32);\young(1)\right), \left(\young(31);\young(2)\right)\right)$ 	&$\bar{2}3\bar{1}$		& $\left(\left(\young(3);\young(12)\right), \left(\young(2);\young(13)\right)\right)$ \\
&&&\\
\hline
&&& \\

$1\bar{3}\bar{2}$		& $\left(\left(\young(1,2);\young(3)\right), \left(\young(1,2);\young(3)\right)\right)$ &$21\bar{3}$		& $\left(\left(\young(1);\young(2,3)\right), \left(\young(1);\young(2,3)\right)\right)$  \\

$321$				& $\left(\left(\young(1,3);\young(2)\right), \left(\young(1,3);\young(2)\right)\right)$ &$\bar{1}2\bar{3}$	& $\left(\left(\young(2);\young(1,3)\right), \left(\young(2);\young(1,3)\right)\right)$  \\

$\bar{3}2\bar{1}$		& $\left(\left(\young(2,3);\young(1)\right), \left(\young(2,3);\young(1)\right)\right)$ &$\bar{1}\bar{2}3$	& $\left(\left(\young(3);\young(1,2)\right), \left(\young(3);\young(1,2)\right)\right)$ \\

$31\bar{2}$			& $\left(\left(\young(1,2);\young(3)\right), \left(\young(1,3);\young(2)\right)\right)$ &$\bar{2}1\bar{3}$	& $\left(\left(\young(1);\young(2,3)\right), \left(\young(2);\young(1,3)\right)\right)$ \\

$2\bar{3}1$			& $\left(\left(\young(1,3);\young(2)\right), \left(\young(1,2);\young(3)\right)\right)$ &$2\bar{1}\bar{3}$	& $\left(\left(\young(2);\young(1,3)\right), \left(\young(1);\young(2,3)\right)\right)$ \\

$\bar{3}1\bar{2}$		& $\left(\left(\young(1,2);\young(3)\right), \left(\young(2,3);\young(1)\right)\right)$ &$\bar{2}\bar{3}1$	& $\left(\left(\young(1);\young(2,3)\right), \left(\young(3);\young(1,2)\right)\right)$ \\

$2\bar{3}\bar{1}$		& $\left(\left(\young(2,3);\young(1)\right), \left(\young(1,2);\young(3)\right)\right)$ &$3\bar{1}\bar{2}$	& $\left(\left(\young(3);\young(1,2)\right), \left(\young(1);\young(2,3)\right)\right)$ \\

$\bar{3}21$			& $\left(\left(\young(1,3);\young(2)\right), \left(\young(2,3);\young(1)\right)\right)$ &$\bar{1}\bar{3}2$	& $\left(\left(\young(2);\young(1,3)\right), \left(\young(3);\young(1,2)\right)\right)$ \\

$32\bar{1}$			& $\left(\left(\young(2,3);\young(1)\right), \left(\young(1,3);\young(2)\right)\right)$ &$\bar{1}3\bar{2}$	& $\left(\left(\young(3);\young(1,2)\right), \left(\young(2);\young(1,3)\right)\right)$ \\
\hline

\end{longtable}}
\normalsize
\section{Examples}

\begin{example} \label{example:partition1}
Consider the bipartition $(\mu,\nu)=((5,4,3,2,1^2);(3,2^4,1))$, for which we picture the space $V$ below.

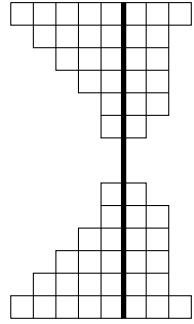
\begin{center}
\begin{tikzpicture}[scale=0.3]
\draw[-,line width=2pt] (0,7) to (0,-7);
\draw (-5,7) -- (-5,6) -- (3,6) -- (3,7) -- (-5,7);
\draw (-4,7) -- (-4,5) -- (2,5); 
\draw (-3,7) -- (-3,4) -- (2,4); 
\draw (-2,7) -- (-2,3) -- (2,3); 
\draw (-1,2) -- (2,2) -- (2,7); 
\draw (-1,7) -- (-1,1) -- (1,1) -- (1,7);

\draw (-5,-7) -- (-5,-6) -- (3,-6) -- (3,-7) -- (-5,-7);
\draw (-4,-7) -- (-4,-5) -- (2,-5); 
\draw (-3,-7) -- (-3,-4) -- (2,-4); 
\draw (-2,-7) -- (-2,-3) -- (2,-3); 
\draw (-1,-2) -- (2,-2) -- (2,-7); 
\draw (-1,-7) -- (-1,-1) -- (1,-1) -- (1,-7); 
\end{tikzpicture}
\end{center}
In this case we have $v=v_{1,5}+v_{2,4}+v_{3,3}+v_{4,2}+v_{5,1}+v_{6,1}$. Let $w=\alpha_{1}v_{1,1} +\alpha_{2}v_{2,1} +\alpha_{3}v_{3,1} +\alpha_{4}v_{4,1} + \beta_{4}v_{4,4}^{*}+\beta_{3}v_{3,5}^{*}+\beta_{2}v_{2,6}^{*}+\beta_{4}v_{1,8}^{*}$ with $\alpha_4 \neq 0$ and $\beta_2+\beta_3+\beta_4=0$. Therefore $\eType(v+W,x_{|W^{\perp}/W})=((5,4,3,1^3);(3,2^4,1))$ - obtained from $(\mu,\nu)$ by decreasing $\mu_4$ by $1$. \vspace{5pt}

We mimic the proof of Proposition \ref{proposition:trick} to confirm that the maximal $k_2$ such that $X \subseteq x^{k_2-1}(W^{\perp})+\C[x]v+w$ is $4(=\lambda_4-1)$ in this case. It is clear that the vectors $v_{1,1}, v_{2,1}, v_{3,1} \in x^{3}(W^{\perp})$ while $v_{4,1} \not\in x^{3}(W^{\perp})$. We now show that   $v_{4,1} \not\in x^{3}(W^{\perp})+\C[x]v$, which implies that $v_{4,4}^{*}$ and therefore $X$ is not contained in $x^{3}(W^{\perp})+\C[x]v+W$. \vspace{5pt}

Now we have $\beta_{3}v_{2,6}-\beta_{2}v_{3,5}$ and  $\beta_{4}v_{3,5}-\beta_{3}v_{4,4} \in W^{\perp}$, which implies that 
$$
\beta_{3}v_{2,3}-\beta_{2}v_{3,2}, \beta_{4}v_{3,2}-\beta_{3}v_{4,1} \in x^{3}(W^{\perp}).
$$ 
Also we have $v_{1,4} \in x^{3}(W^{\perp})$ and so we require the vector $x(v)-v_{1,4}=v_{2,3}+v_{3,2}+v_{4,1}$ to be linearly independent from $\beta_{3}v_{2,3}-\beta_{2}v_{3,2} $ and  $\beta_{4}v_{3,2}-\beta_{3}v_{4,1}.$ However 
$$
\beta_{3}v_{2,3}-\beta_{2}v_{3,2} - (\beta_{4}v_{3,2}-\beta_{3}v_{4,1} )=\beta_{3}v_{2,3}+(-\beta_{2} -\beta_{4})v_{3,2}+\beta_{3}v_{4,1}=\beta_{3} (v_{2,3}+v_{3,2}+v_{4,1})
$$ 
since $\beta_2+\beta_3+\beta_4=0$. Therefore it is impossible to isolate $v_{4,1}$ in $x^{3}(W^{\perp})+\C[x]v$, which proves that $k_2=3$ in this case. Also by Proposition \ref{proposition:genericl2}, we know that $l_2=3$ and since $5 = \max \Delta_4 < \max \Gamma_5=6$ we decrease $\nu_{5}=\nu_{\max \Delta_4}$ by $1$ to obtain $\eType(v+Y,x_{|Y^{\perp}/Y})=((5,4,3,1^3),(3,2^3,1^2))$. Pictorially we have 

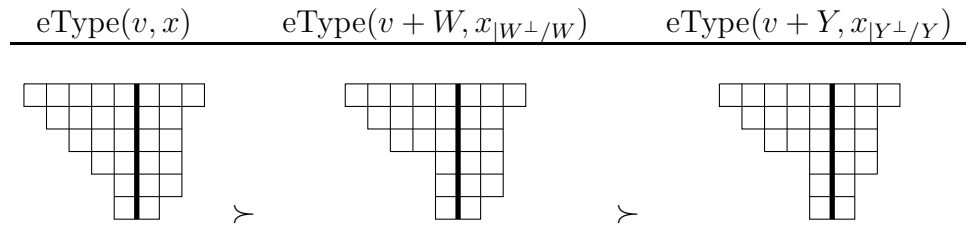
\begin{center}
$$
\begin{array}{ccccc}
 \eType(v,x)& & \eType(v+W,x_{|W^{\perp}/W}) && \eType(v+Y,x_{|Y^{\perp}/Y})   \\
\hline\\

\begin{tikzpicture}[scale=0.3]

\draw[-,line width=2pt] (0,7) to (0,1);
\draw (-5,7) -- (-5,6) -- (3,6) -- (3,7) -- (-5,7);
\draw (-4,7) -- (-4,5) -- (2,5); 
\draw (-3,7) -- (-3,4) -- (2,4); 
\draw (-2,7) -- (-2,3) -- (2,3); 
\draw (-1,2) -- (2,2) -- (2,7); 
\draw (-1,7) -- (-1,1) -- (1,1) -- (1,7);

\end{tikzpicture}

&
\succ
&
\begin{tikzpicture}[scale=0.3]
\draw[-,line width=2pt] (0,7) to (0,1);
\draw (-5,7) -- (-5,6) -- (3,6) -- (3,7) -- (-5,7);
\draw (-4,7) -- (-4,5) -- (2,5); 
\draw (-3,7) -- (-3,4) -- (2,4); 
\draw (-2,7) -- (-2,4);  
\draw (-1,3) -- (2,3);
\draw (-1,2) -- (2,2) -- (2,7); 
\draw (-1,7) -- (-1,1) -- (1,1) -- (1,7); 

\end{tikzpicture}
&
\succ
&
\begin{tikzpicture}[scale=0.3]
\draw[-,line width=2pt] (0,7) to (0,1);
\draw (-5,7) -- (-5,6) -- (3,6) -- (3,7) -- (-5,7);
\draw (-4,7) -- (-4,5) -- (2,5); 
\draw (-3,7) -- (-3,4) -- (2,4); 
\draw (-2,7) -- (-2,4);  
\draw (-1,3) -- (2,3);
\draw (-1,2) -- (1,2);
\draw (2,3) -- (2,7); 
\draw (-1,7) -- (-1,1) -- (1,1) -- (1,7); 

\end{tikzpicture}

\end{array}
$$
\end{center}

\end{example}

\begin{example}
Let us vary the partitions in Example \ref{example:partition1} slightly to $(\mu,\nu)=((5,4,3,2,1);(3,2^4,1))$, but keeping $w$ the same. Then our picture would look like: 

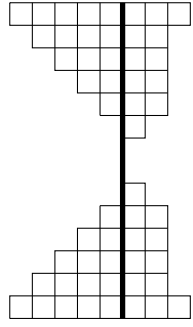
\begin{center}
\begin{tikzpicture}[scale=0.3]
\draw[-,line width=2pt] (0,7) to (0,-7);
\draw (-5,7) -- (-5,6) -- (3,6) -- (3,7) -- (-5,7);
\draw (-4,7) -- (-4,5) -- (2,5); 
\draw (-3,7) -- (-3,4) -- (2,4); 
\draw (-2,7) -- (-2,3) -- (2,3); 
\draw (-1,2) -- (2,2) -- (2,7); 
\draw (-1,7) -- (-1,2);
\draw (0,1) -- (1,1) -- (1,7);

\draw (-5,-7) -- (-5,-6) -- (3,-6) -- (3,-7) -- (-5,-7);
\draw (-4,-7) -- (-4,-5) -- (2,-5); 
\draw (-3,-7) -- (-3,-4) -- (2,-4); 
\draw (-2,-7) -- (-2,-3) -- (2,-3); 
\draw (-1,-2) -- (2,-2) -- (2,-7); 
\draw (-1,-7) -- (-1,-2);
\draw (0,-1) -- (1,-1) -- (1,-7); 

\end{tikzpicture}
\end{center}
Again we would have $m=4$ and $\eType(v+W,x_{|W^{\perp}/W})=((5,4,3,1^2);(3,2^4,1))$, obtained from $(\mu,\nu)$ by decreasing $\mu_4$ by $1$. Exactly the same calculations performed in Example  \ref{example:partition1} again show that $k_2=l_2=4$ but this time we have that $\max \Gamma_5= \max \Delta_{4}=5$. So $\eType(v+Y,x{|Y^{\perp}/Y})$ would be obtained from $\eType(v+W,x{|W^{\perp}/W})$ by decreasing $\mu_5=\mu_{\max \Gamma_{5}}$ by $1$ and the succession of exotic types are as follows: 

$$
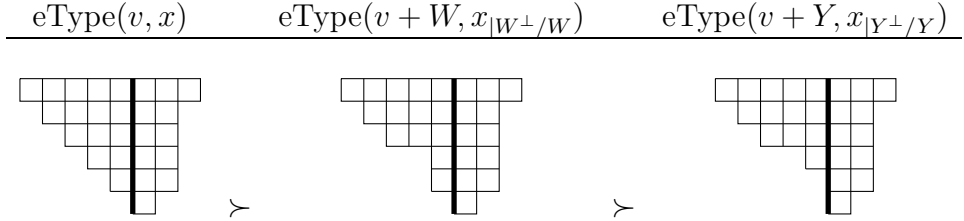
\begin{array}{ccccc}
 \eType(v,x)& & \eType(v+W,x_{|W^{\perp}/W}) && \eType(v+Y,x_{|Y^{\perp}/Y})   \\
\hline\\

\begin{tikzpicture}[scale=0.3]
\draw[-,line width=2pt] (0,7) to (0,1);
\draw (-5,7) -- (-5,6) -- (3,6) -- (3,7) -- (-5,7);
\draw (-4,7) -- (-4,5) -- (2,5); 
\draw (-3,7) -- (-3,4) -- (2,4); 
\draw (-2,7) -- (-2,3) -- (2,3); 
\draw (-1,2) -- (2,2) -- (2,7); 
\draw (-1,7) -- (-1,2);
\draw (0,1) -- (1,1) -- (1,7); 

\end{tikzpicture}

&
\succ
&

\begin{tikzpicture}[scale=0.3]
\draw[-,line width=2pt] (0,7) to (0,1);
\draw (-5,7) -- (-5,6) -- (3,6) -- (3,7) -- (-5,7);
\draw (-4,7) -- (-4,5) -- (2,5); 
\draw (-3,7) -- (-3,4) -- (2,4); 
\draw (-2,7) -- (-2,4); 
\draw (-1,3) -- (2,3);
\draw (-1,2) -- (2,2) -- (2,7); 
\draw (-1,7) -- (-1,2);
\draw (0,1) -- (1,1) -- (1,7); 

\end{tikzpicture}

&
\succ
&

\begin{tikzpicture}[scale=0.3]
\draw[-,line width=2pt] (0,7) to (0,1);
\draw (-5,7) -- (-5,6) -- (3,6) -- (3,7) -- (-5,7);
\draw (-4,7) -- (-4,5) -- (2,5); 
\draw (-3,7) -- (-3,4) -- (2,4); 
\draw (-2,7) -- (-2,4); 
\draw (-1,3) -- (2,3);
\draw (0,2) -- (2,2) -- (2,7); 
\draw (-1,7) -- (-1,3);
\draw (0,1) -- (1,1) -- (1,7); 

\end{tikzpicture}

\end{array}
$$

\end{example}

\begin{example} \label{example:5.2.2}
Let $\mu=(3,2^3)$ and $\nu=(3,2^2,1)$. 

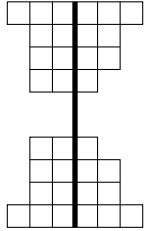
\begin{center}
\begin{tikzpicture}[scale=0.3]
\draw[-,line width=2pt] (0,5) to (0,-5);
\draw (-3,5) -- (-3,4) -- (3,4) -- (3,5) -- (-3,5);
\draw (-2,5) -- (-2,1) -- (1,1) -- (1,5);
\draw (-1,5) -- (-1,1);
\draw (-2,3) -- (2,3);
\draw (-2,2) --  (2,2) -- (2,5);

\draw (-3,-5) -- (-3,-4) -- (3,-4) -- (3,-5) -- (-3,-5);
\draw (-2,-5) -- (-2,-1) -- (1,-1) -- (1,-5);
\draw (-1,-5) -- (-1,-1);
\draw (-2,-3) -- (2,-3);
\draw (-2,-2) --  (2,-2) -- (2,-5);
\end{tikzpicture}
\end{center}
Let 
$$
w=\alpha_{1}v_{1,1}+\alpha_{2}v_{2,1}+\alpha_{3}v_{3,1}+v_{4,1}+\beta_{3}v_{3,4}^{*}+\beta_{2}v_{2,4}^{*}+\beta_{1}v_{1,6}^{*},
$$ 
so that $\eType(v+W, x_{|W^{\perp}/W})=((3,2^2,1);(3,2^2,1))$ obtained from $(\mu,\nu)$ by decreasing $\mu_4$ by $1$. Observe that in this case we have $m=4$ and $2=\nu_{3}>\nu_4=1$.\vspace{5pt}

By Proposition \ref{proposition:D1k2} $k_2$ such that $X \subseteq x^{k_2-1}(W^{\perp}) + \C[x]v + W$ is $4=\lambda_3$ and the maximal $l_2$ such that $X^{\perp}\supseteq (x^{l_2-1})^{-1}(\C[x]v+X+W) \cap W^{\perp}$ is $3(=\mu_4+\nu_3-1)$. We give a few more details as an aide to the proof of Proposition \ref{proposition:D1l2}. We have $X^{\perp} \supseteq x^{-2}(W)$ since $X$ and $W$ are not supported on $v_{4,1}^{*}$, this implies that $X^{\perp} \supseteq  x^{-2}(\C[x]v+X+W)\cap W^{\perp}.$ Also, $X^{\perp} \subseteq \ker(x^3) \cap W^{\perp}$ and so $l_2 \geq 3.$ \vspace{5pt}

The vectors $u, u'$ and $u''$ in Proposition  \ref{proposition:D1k2} are 
\begin{eqnarray*}
u	&=& v_{1,2}+(1-\alpha_2)v_{2,1}+(1-\alpha_3)v_{3,1}-(\alpha_{1}v_{1,1}+\beta_{3}v_{3,4}^{*}+\beta_{2}v_{2,4}^{*}+\beta_{1}v_{1,6}^{*}) = x(v)-w \\
u'	&=& v_{1,5}+(1-\alpha_2)v_{2,4}+(1-\alpha_3)v_{3,4}-(\alpha_{1}v_{1,4}+\beta_{3}v_{3,1}^{*}+\beta_{2}v_{2,1}^{*}+\beta_{1}v_{1,3}^{*} )\\
u''	&=& u'+(\beta_2+\beta_3)v_{3,1}^{*}.
\end{eqnarray*}
Note that $x^{3}(u'')=u$, $\langle w,u' \rangle=\beta_2+\beta_3$ and so $\langle w,u'' \rangle=0$. Therefore $u''  \in x^{-3}(\C[x]v+X+W)\cap W^{\perp}$, but $\langle \ttx,u''\rangle=(1-\alpha_2)\delta_2+(1-\alpha_3)\delta_3-(-\beta_2\gamma_2-\beta_3\gamma_3+\beta_2+\beta_3),$ which is generically non-zero. Therefore $X^{\perp} \not\supseteq x^{-3}(\C[x]v+X+W)\cap W^{\perp}$ and so $l_2=3$. \vspace{5pt}

Therefore, $\eType(v+Y,x_{|Y^{\perp}/Y})=((3,2,2,1);(3,2,1,1))$. Pictorially we have 
\begin{center}
$$
\begin{array}{ccccc}
 \eType(v,x)& & \eType(v+W,x_{|W^{\perp}/W}) && \eType(v+Y,x_{|Y^{\perp}/Y})   \\
\hline\\

\begin{tikzpicture}[scale=0.3]
\draw[-,line width=2pt] (0,5) to (0,1);
\draw (-3,5) -- (-3,4) -- (3,4) -- (3,5) -- (-3,5);
\draw (-2,5) -- (-2,1) -- (1,1) -- (1,5);
\draw (-1,5) -- (-1,1);
\draw (-2,3) -- (2,3);
\draw (-2,2) --  (2,2) -- (2,5);

\end{tikzpicture}

&
\succ
&
\begin{tikzpicture}[scale=0.3]
\draw[-,line width=2pt] (0,5) to (0,1);
\draw (-3,5) -- (-3,4) -- (3,4) -- (3,5) -- (-3,5);
\draw (-2,5) -- (-2,2);
\draw (-1,1)--(1,1) -- (1,5);
\draw (-1,5) -- (-1,1);
\draw (-2,3) -- (2,3);
\draw (-2,2) --  (2,2) -- (2,5);
\end{tikzpicture}
&
\succ
&
\begin{tikzpicture}[scale=0.3]
\draw[-,line width=2pt] (0,5) to (0,1);
\draw (-3,5) -- (-3,4) -- (3,4) -- (3,5) -- (-3,5);
\draw (-2,5) -- (-2,2);
\draw (-1,1)--(1,1) -- (1,5);
\draw (-1,5) -- (-1,1);
\draw (-2,3) -- (2,3) -- (2,5);
\draw (-2,2) --  (1,2);
\end{tikzpicture}

\end{array}
$$
\end{center}

\end{example}

\begin{example} \label{5.2.1} Let $\mu=(4,3,3,2,2)$ and $\nu=(4,3,3,1)$, so $V$ can be pictured as follows: 

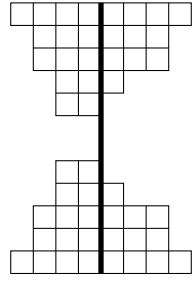
\begin{center}
\begin{tikzpicture}[scale=0.3]
\draw[-,line width=2pt] (0,6) to (0,-6);
\draw (-4,6) -- (-4,5) -- (4,5) -- (4,6) -- (-4,6);
\draw (-3,6) -- (-3,3) -- (3,3) -- (3,6);
\draw (-3,4) -- (3,4);
\draw (2,6) -- (2,3);
\draw (-2,6) -- (-2,1) -- (0,1); 
\draw (-2,2) -- (1,2) -- (1,6); 
\draw (-1,6) -- (-1,1) ; 

\draw (-4,-6) -- (-4,-5) -- (4,-5) -- (4,-6) -- (-4,-6);
\draw (-3,-6) -- (-3,-3) -- (3,-3) -- (3,-6);
\draw (-3,-4) -- (3,-4);
\draw (2,-6) -- (2,-3);
\draw (-2,-6) -- (-2,-1) -- (0,-1); 
\draw (-2,-2) -- (1,-2) -- (1,-6); 
\draw (-1,-6) -- (-1,-1) ;

\end{tikzpicture}
\end{center}
Here $v=v_{1,4}+v_{2,3}+v_{3,3}+v_{4,2}+v_{5,2}$. Let $$w=\alpha_{1}v_{1,1}+\alpha_{2}v_{2,1}+\alpha_{3}v_{3,1}+\beta_{3}v_{3,6}^{*}+\beta_{2}v_{2,6}^{*}+\beta_{1}v_{1,8}^{*},$$ with $\beta_2 \neq -\beta_3$ so that $\eType(v+W,x_{|W^{\perp}/W})=((4,3,3,2,2);(4,3,2,1))$. In this case we have $m=3$ and $\nu_2=\nu_3=2$. We will explicitly confirm the assertion in Proposition \ref{proposition:D2k2} that the maximal $k_2$ such that $X \subseteq x^{k-1}(W^{\perp}) + \C[x]v + W$ is $6=\lambda_3$. \vspace{5pt}

Firstly observe that $v_{1,1}$ and $v_{1,8}^{*}$ are clearly contained in $x^5(W^{\perp})$ as is  $v_{1,2}=x^5(v_{1,7}).$ We also have $\beta_{2}v_{3,6}-\beta_{3}v_{2,6} \in W^{\perp}$ and so $\beta_{2}v_{3,1}-\beta_{3}v_{2,1} \in x^{5}(W^{\perp})$. Hence $v_{2,1}+v_{3,1}=x^2(v)-v_{1,2} \in x^5(W^{\perp})+\C[x]v$ and since $\beta_2 \neq -\beta_3$, it follows that $\beta_{2}v_{3,1}-\beta_{3}v_{2,1}$ and $v_{2,1}+v_{3,1}$ are not scalar multiples of each other. Therefore we conclude that $v_{2,1}$ and $v_{3,1}$ are contained in $x^5(W^{\perp})+\C[x]v$.\vspace{5pt}

Next observe that $\alpha_{2}v_{3,1}^{*}-\alpha_{3}v_{2,1}^{*} \in W^{\perp}$ and so $\alpha_{2}v_{3,6}^{*}-\alpha_{3}v_{2,6}^{*} \in x^{5}(W^{\perp})$. Therefore we have 
$$
\C\{v_{1,1},v_{2,1}, v_{3,1},v_{1,8}^{*}, \alpha_{2}v_{3,6}^{*}-\alpha_{3}v_{2,6}^{*} \} \subseteq x^5(W^{\perp})+\C[x]v
$$ 
and so 
$$
\C\{v_{1,1},v_{2,1}, v_{3,1},v_{1,8}^{*}, \alpha_{2}v_{3,6}^{*}-\alpha_{3}v_{2,6}^{*} \} +W \subseteq x^5(W^{\perp})+\C[x]v +W. 
$$ 
Since $W$ is chosen generically, we may assume that $\alpha_{2}v_{3,6}^{*}-\alpha_{3}v_{2,6}^{*}$ and $\beta_{3}v_{3,6}^{*}+\beta_{2}v_{2,6}^{*}$ are not scalar multiples of each other and so we have that $v_{3,6}^{*}$ and $v_{2,6}^{*}$ are contained in $x^5(W^{\perp})+\C[x]v+W= x^5(W^{\perp})+\C[x]v$, since $W \subseteq x^5(W^{\perp})$. Therefore $X \subseteq x^5(W^{\perp})+\C[x]v+W$ which confirms Proposition \ref{proposition:D2k2} in that $k_2=6$.\vspace{5pt}

Next we easily see that $X^{\perp} \supseteq x^{-4}(\C[x]v+X+W)\cap W^{\perp}$, but $X^{\perp} \not\supseteq x^{-5}(W)$ and so $l_2=5$. (Note that $X^{\perp} \not\supseteq x^{-5}(\C[x]v)$ but  $X^{\perp} \supseteq x^{-5}(\C[x]v)\cap W^{\perp}$). Therefore we have $\eType(v+Y,x_{|Y^{\perp}/Y})=((4,3,2,2,2);(4,3,2,1))$. Pictorially, we have 

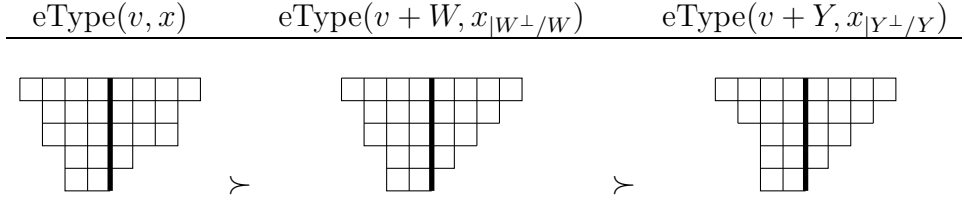
\begin{center}
$$
\begin{array}{ccccc}
 \eType(v,x)& & \eType(v+W,x_{|W^{\perp}/W}) && \eType(v+Y,x_{|Y^{\perp}/Y})   \\
\hline\\

\begin{tikzpicture}[scale=0.3]
\draw[-,line width=2pt] (0,6) to (0,1);
\draw (-4,6) -- (-4,5) -- (4,5) -- (4,6) -- (-4,6);
\draw (-3,6) -- (-3,3) -- (3,3) -- (3,6);
\draw (-3,4) -- (3,4);
\draw (2,6) -- (2,3);
\draw (-2,6) -- (-2,1) -- (0,1); 
\draw (-2,2) -- (1,2) -- (1,6); 
\draw (-1,6) -- (-1,1) ; 

\end{tikzpicture}

&
\succ
&
\begin{tikzpicture}[scale=0.3]
\draw[-,line width=2pt] (0,6) to (0,1);
\draw (-4,6) -- (-4,5) -- (4,5) -- (4,6) -- (-4,6);
\draw (-3,6) -- (-3,3) -- (2,3);
\draw (3,6) -- (3,4);
\draw (-3,4) -- (3,4);
\draw (2,6) -- (2,3);
\draw (-2,6) -- (-2,1) -- (0,1); 
\draw (-2,2) -- (1,2) -- (1,6); 
\draw (-1,6) -- (-1,1) ; 

\end{tikzpicture}
&
\succ
&
\begin{tikzpicture}[scale=0.3]
\draw[-,line width=2pt] (0,6) to (0,1);
\draw (-4,6) -- (-4,5) -- (4,5) -- (4,6) -- (-4,6);
\draw (-3,6) -- (-3,4);
\draw (3,6) -- (3,4);
\draw (-3,4) -- (3,4);
\draw (2,6) -- (2,3) -- (-2,3);
\draw (-2,6) -- (-2,1) -- (0,1); 
\draw (-2,2) -- (1,2) -- (1,6); 
\draw (-1,6) -- (-1,1) ; 
\end{tikzpicture}

\end{array}
$$
\end{center}

\end{example}

\begin{example} \label{5.2.3}
Consider the bipartition $(\mu,\nu)=((2^2);(3,2,1))$. 

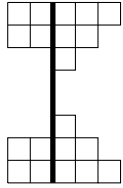
\begin{center}
\begin{tikzpicture}[scale=0.3]
\draw[-,line width=2pt] (0,4) to (0,-4);
\draw (-2,4) -- (-2,2) -- (2,2) -- (2,4) -- (-2,4);
\draw (-1,4) -- (-1,2);
\draw (-2,3) -- (3,3) -- (3,4) -- (2,4);
\draw (1,4) -- (1,1) -- (0,1);

\draw (-2,-4) -- (-2,-2) -- (2,-2) -- (2,-4) -- (-2,-4);
\draw (-1,-4) -- (-1,-2);
\draw (-2,-3) -- (3,-3) -- (3,-4) -- (2,-4);
\draw (1,-4) -- (1,-1) -- (0,-1);

\end{tikzpicture}
\end{center}

In this case we have $v=v_{1,2}+v_{2,2}$. Let $w=\alpha_{1}v_{1,1} + \beta_{1}v_{1,5}^{*}$ so that $m=1$ and we have  $\eType(v+W,x_{|W^{\perp}/W})=((2^2);(2^2,1))$ - obtained from $(\mu,\nu)$ by decreasing $\nu_1$ by $1$. Also since $2=\max \Gamma_1 > \max \Delta_1=1$, Proposition \ref{proposition:k2d0} tells us $k_2=4(=\lambda_1-1)$. We confirm this explicitly.  It is clear that both $v_{1,1}$ and $v_{1,5}^{*}$ are both contained in $x^{3}(W^{\perp})$ but not $x^{4}(W^{\perp})=W$. So if $X$ were contained in $x^{4}(W^{\perp})+\C[x]v+W$, it would be necessary that $X \subseteq \C[x]v+W= \C\{v_{1,2}+v_{2,2}, v_{1,1}+v_{2,1}, \alpha_{1}v_{1,1} + \beta_{1}v_{1,5}^{*}\}$, which is clearly impossible. Therefore, $k_2=4$ as predicted. \vspace{5pt}

By Proposition \ref{proposition:genericl2}, we have $l_2=4$ as well and so $\eType(w+Y,x_{|Y^{\perp}/Y})=  ((2,1);(2^2,1)$ obtained from $((2^2);(2^2,1))$ by decreasing $\mu_2=\mu_{\max \Gamma_1}$ by $1$. Thus:

$$
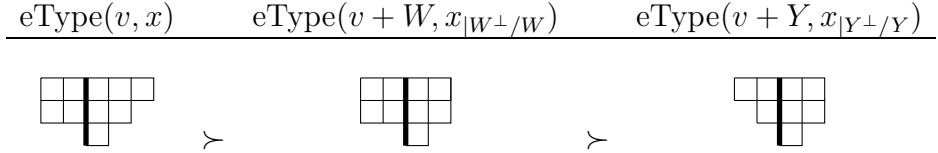
\begin{array}{ccccc}
 \eType(v,x)& & \eType(v+W,x_{|W^{\perp}/W}) && \eType(v+Y,x_{|Y^{\perp}/Y})   \\
\hline\\

\begin{tikzpicture}[scale=0.3]
\draw[-,line width=2pt] (0,4) to (0,1);
\draw (-2,4) -- (-2,2) -- (2,2) -- (2,4) -- (-2,4);
\draw (-1,4) -- (-1,2);
\draw (-2,3) -- (3,3) -- (3,4) -- (2,4);
\draw (1,4) -- (1,1) -- (0,1);

\end{tikzpicture}

&
\succ
&

\begin{tikzpicture}[scale=0.3]
\draw[-,line width=2pt] (0,4) to (0,1);
\draw (-2,4) -- (-2,2) -- (2,2) -- (2,4) -- (-2,4);
\draw (-1,4) -- (-1,2);
\draw (-2,3) -- (2,3) -- (2,4);
\draw (1,4) -- (1,1) -- (0,1);
\end{tikzpicture}
&
\succ
&

\begin{tikzpicture}[scale=0.3]
\draw[-,line width=2pt] (0,4) to (0,1);
\draw (-2,3) -- (2,3) -- (2,4) -- (-2,4);
\draw (-2,4) -- (-2,3);

\draw (-1,4) -- (-1,2);
\draw (-1,2) -- (2,2) -- (2,3);
\draw (1,4) -- (1,1) -- (0,1);

\end{tikzpicture}

\end{array}
$$
\end{example}

\begin{example}
If we were to change the above bipartition slightly to $(\mu,\nu)=((2^3);(3,2,1))$: 
\begin{center}
\begin{tikzpicture}[scale=0.3]
\draw[-,line width=2pt] (0,4) to (0,-4);
\draw (-2,4) -- (-2,1) -- (0,1);
\draw (2,4) -- (-2,4);
\draw (-1,4) -- (-1,1);
\draw (-2,3) -- (3,3) -- (3,4) -- (2,4);
\draw (-2,2) -- (2,2) -- (2,4);
\draw (1,4) -- (1,1) -- (0,1);

\draw (-2,-4) -- (-2,-1) -- (0,-1);
\draw (2,-4) -- (-2,-4);
\draw (-1,-4) -- (-1,-1);
\draw (-2,-3) -- (3,-3) -- (3,-4) -- (2,-4);
\draw (-2,-2) -- (2,-2) -- (2,-4);
\draw (1,-4) -- (1,-1) -- (0,-1);

\end{tikzpicture}
\end{center}
Then we would still have $m=1$ and exactly the same calculations would show that $k_2=l_2=4$. But in this case we would have $3 = \max \Gamma_1 > \max \Delta_{2}=2$, that is $ \max \Gamma_m > \max \Delta_{m+1}=2$ and so we would obtain $\eType(w+Y,x_{|Y^{\perp}/Y})$ by decreasing $\nu_{2}=\nu_{\max \Delta_2}$ by $1$, and the succession of exotic types would be as follows:

$$
\begin{array}{ccccc}
 \eType(v,x)& & \eType(v+W,x_{|W^{\perp}/W}) && \eType(v+Y,x_{|Y^{\perp}/Y})   \\
\hline\\

\begin{tikzpicture}[scale=0.3]
\draw[-,line width=2pt] (0,4) to (0,1);
\draw (-2,4) -- (-2,1) -- (0,1);
\draw (2,4) -- (-2,4);
\draw (-1,4) -- (-1,1);
\draw (-2,3) -- (3,3) -- (3,4) -- (2,4);
\draw (-2,2) -- (2,2) -- (2,4);
\draw (1,4) -- (1,1) -- (0,1);

\end{tikzpicture}

&
\succ
&

\begin{tikzpicture}[scale=0.3]
\draw[-,line width=2pt] (0,4) to (0,1);
\draw (-2,4) -- (-2,1) -- (0,1);
\draw (2,4) -- (-2,4);
\draw (-1,4) -- (-1,1);
\draw (-2,3) -- (2,3);  
\draw (-2,2) -- (2,2) -- (2,4);
\draw (1,4) -- (1,1) -- (0,1);

\end{tikzpicture}

&
\succ
&

\begin{tikzpicture}[scale=0.3]
\draw[-,line width=2pt] (0,4) to (0,1);
\draw (-2,4) -- (-2,1) -- (0,1);
\draw (2,4) -- (-2,4);
\draw (-1,4) -- (-1,1);
\draw (-2,3) -- (2,3);  
\draw (-2,2) -- (1,2);
\draw (2,3) -- (2,4);
\draw (1,4) -- (1,1) -- (0,1);

\end{tikzpicture}

\end{array}
$$

\end{example}

\end{appendix}


\def\cprime{$'$} \newcommand{\arxiv}[1]{\href{http://arxiv.org/abs/#1}{\tt
  arXiv:\nolinkurl{#1}}}

\end{document}